\documentclass[a4paper,1	2pt]{amsart}
\usepackage{amsmath,amssymb,mathtools} 
\usepackage[british]{babel}
\usepackage[final,nopatch=footnote]{microtype}
\usepackage{mlmodern}
\usepackage[all,cmtip]{xy}
\usepackage[margin=1in]{geometry}
\usepackage[UKenglish]{isodate}
\usepackage[shortlabels]{enumitem}
\usepackage[hidelinks]{hyperref}
\usepackage{epigraph}
\usepackage{footnote}
\usepackage[T1]{fontenc}
\usepackage{bm}
\usepackage{mathrsfs}

\swapnumbers
\date{2nd June 2026}
\makeatletter
\let\c@subsection\c@equation
\makeatother
\numberwithin{equation}{section}

\theoremstyle{plain}
\newtheorem{theorem}[subsection]{Theorem}
\newtheorem{proposition}[subsection]{Proposition}
\newtheorem{lemma}[subsection]{Lemma}
\newtheorem{corollary}[subsection]{Corollary}

\theoremstyle{definition} 
\newtheorem{definition}[subsection]{Definition}
\newtheorem{example}[subsection]{Example}

\theoremstyle{remark}
\newtheorem{remark}[subsection]{Remark}

\newcommand{\cd}[2][]{\vcenter{\hbox{\xymatrix#1{#2}}}}

\newcommand{\hdash}{\rotatebox[origin=c]{90}{$\vdash$}}
\newcommand{\vcong}{\rotatebox[origin=c]{-90}{$\cong$}}
\newcommand{\fatpullbackcorner}[1][dr]{\save*!/#1-1.5pc/#1:(-1,1)@^{|-}\restore}
\newcommand{\fatterpullbackcorner}[1][dr]{\save*!/#1-2pc/#1:(-1,1)@^{|-}\restore}
\newcommand{\fatpushoutcorner}[1][dr]{\save*!/#1+1.75pc/#1:(1,-1)@^{|-}\restore}
\newcommand{\dtwocell}[3][0.5]{\ar@{}[#2] \ar@{=>}?(#1)+/u  0.2cm/;?(#1)+/d 0.2cm/^{#3}}
\newcommand{\dltwocell}[3][0.5]{\ar@{}[#2] \ar@{=>}?(#1)+/ur  0.2cm/;?(#1)+/dl 0.2cm/^{#3}}
\newcommand{\dtwocellleft}[3][0.5]{\ar@{}[#2] \ar@{=>}?(#1)+/u  0.2cm/;?(#1)+/d 0.2cm/_{#3}}
\newcommand{\ltwocell}[3][0.5]{\ar@{}[#2] \ar@{=>}?(#1)+/r 0.2cm/;?(#1)+/l 0.2cm/_{#3}}

\newdir{ >}{{}*!/-5pt/@{>}}

\usepackage{scalerel}

\DeclareMathOperator*{\bigboxtimes}{\scalerel*{\boxtimes}{\prod}}

\title{A convenient model category for bicategories}
\author{Alexander Campbell}
\address{School of Mathematics and Statistics F07 \\ University of Sydney \\ NSW 2006 \\ Australia}
\email{campbell@maths.usyd.edu.au}
\urladdr{https://acmbl.github.io/}

\begin{document}

\begin{abstract}
We introduce and study the model category of flexible $2$\nobreakdash-categories, which is Quillen equivalent to Lack's  model category of $2$-categories, but enjoys several excellent properties not shared by the latter. In particular, every object of this model category is cofibrant, it is a monoidal model category with respect to its cartesian closed structure, and its full subcategory of fibrant objects is equivalent to the category of bicategories and normal pseudofunctors.
\end{abstract}
\maketitle
\epigraph{\flushright{And here's a marvellous convenient place for our rehearsal.}}{--- Shakespeare, \textit{A Midsummer Night's Dream} III.i}
\tableofcontents

\section{Introduction}
It is a commonplace of higher category theory that, while every weak $2$-category is equivalent to a strict $2$-category, not every weak $3$-category is equivalent to a strict $3$-category. Nevertheless, the coherence theorem of Gordon, Power, and Street \cite{GPS} states that every tricategory is triequivalent to a $\mathbf{Gray}$-category, that is, to a category enriched over  the category $\bm{2}\textbf{Cat}$ of (strict) $2$-categories and (strict) $2$-functors equipped with Gray's symmetric monoidal structure \cite{gray}. Furthermore, one can, to a large extent, model  the category theory of tricategories by homotopy coherent $\mathbf{Gray}$-enriched category theory, that is, category theory enriched over $\bm{2}\textbf{Cat}$ as a monoidal model category, equipped with both Gray's symmetric monoidal structure and Lack's model structure \cite{Lac02, Lac04}.

However, this $\mathbf{Gray}$-enriched model  for three-dimensional category theory is not entirely satisfactory, as several important constructions on $2$-categories fail to define $\mathbf{Gray}$-enriched functors. Two illustrative examples are the pseudofunctor classifier of a $2$-category and the free $\mathbf{Gray}$-monoid on a $2$-category (cf.\ \cite[p.\ 11]{Gur13}), whose failures to define $\mathbf{Gray}$-enriched functors are reflections of the following two properties of the monoidal model category $\bm{2}\textbf{Cat}$:\  that not every $2$-category is cofibrant in Lack's model structure (cf.\ \cite[\S A]{LR16}), and  that Gray's symmetric monoidal structure is not the cartesian monoidal structure (cf.\ \cite[\S 3]{Shu12}). These two properties therefore make the monoidal model category $\bm{2}\textbf{Cat}$ an inconvenient base of enrichment for three-dimensional categorical homotopy theory and three-dimensional universal algebra.

The purpose of this paper is to introduce a more convenient base of enrichment for three-dimensional category theory:\ the cartesian model category of  flexible $2$-categories. This model category is Quillen equivalent to Lack's model category of $2$-categories, but it has the desired properties that all of its objects are cofibrant and that it is a monoidal model category with respect to its cartesian closed structure. Remarkably, its full subcategory of fibrant objects is equivalent to the category of bicategories and normal pseudofunctors, making it also a more convenient model category for studying bicategories than Lack's model category of bicategories \cite{Lac04}, in which one has direct access only to the less natural strict functors of bicategories.

Following a recollection of the category of free categories in \S\ref{freecat}, we  
 begin our study proper  in \S\ref{algcof} by introducing the category $\textbf{Flex}$ of flexible $2$-categories. In \S\ref{modstr} we construct a model structure on this category which is left-induced from Lack's model structure on $\bm{2}\textbf{Cat}$, whose fibrant objects and fibrations with fibrant codomain we characterise in \S\ref{fibob} and \S\ref{fibsec} respectively. Finally,  in \S\ref{monenr} we prove that this model structure is monoidal with respect to both the restriction to $\mathbf{Flex}$ of Gray's symmetric monoidal structure and the cartesian closed structure of $\mathbf{Flex}$.

\section{Preliminaries on free categories} \label{freecat}
In this preliminary section, we recollect those notions and results pertaining to free categories that we shall use in the sequel. 

\subsection{The category of free categories}
A morphism $f$ in a category is \emph{atomic} if it is not an identity and if, whenever $f = h\circ g$, either $g$ or $h$ is an identity. A category is \emph{free} if each of its morphisms $f$ can be uniquely expressed as a composite of atomic morphisms $f = f_n \circ \cdots \circ f_1$ ($n \geq 0$). A functor  between free categories is a \emph{morphism of free categories} if it sends each atomic morphism in the domain to an atomic morphism or an identity morphism in the codomain. 

This defines the category $\mathbf{Free}$ of free categories as a non-full subcategory of $\mathbf{Cat}$. This subcategory is \emph{replete}, meaning that any category isomorphic to a free category is free and that any isomorphism in $\mathbf{Cat}$ between free categories is a morphism of free categories; in short, this means that the subcategory inclusion $\mathbf{Free} \longrightarrow \mathbf{Cat}$ is an isofibration.

Let $\mathbf{Gph}$ denote the category of reflexive graphs. The forgetful functor $\mathbf{Cat} \longrightarrow \mathbf{Gph}$ admits a left adjoint $\mathbf{Gph} \longrightarrow \mathbf{Cat}$ which sends a reflexive graph $X$ to the free category whose objects are the vertices of $X$, whose morphisms are the paths of non-degenerate edges in $X$, whose  composition operation is concatenation of paths, and whose  identity and atomic morphisms are  the paths of length zero and one respectively. This left adjoint functor factors through an equivalence of categories $\mathbf{Gph} \simeq \mathbf{Free}$. 

\subsection{Colimits and limits in $\mathbf{Free}$}
The category $\mathbf{Free}$ therefore enjoys all the categorical properties of a presheaf category; in particular, it admits all small colimits and limits. Moreover, the inclusion $\mathbf{Free} \longrightarrow \mathbf{Cat}$ is a left adjoint functor and hence preserves all colimits; it also preserves coreflexive equalisers, as may be shown by a straightforward calculation.  

To distinguish it from the cartesian product in $\mathbf{Cat}$, we shall denote the product in $\mathbf{Free}$ of a family of free categories $(C_i)_{i \in I}$ by $\bigboxtimes_{i \in I} C_i$. For example, the product $\mathbb{D}^1 \boxtimes \mathbb{D}^1$ in $\mathbf{Free}$, where $\mathbb{D}^1 = \{\bullet \longrightarrow \bullet\}$ denotes the free category containing a morphism, is the free category on the following graph.
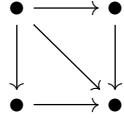
\begin{equation*}
\cd{
\bullet \ar[r] \ar[d] \ar[dr] & \bullet \ar[d] \\
\bullet \ar[r] & \bullet
}
\end{equation*}
Note that the terminal category $\mathbb{D}^0$ is also the terminal object of $\mathbf{Free}$.

\subsection{The path category functor}
We denote the right adjoint of the inclusion functor $\mathbf{Free} \longrightarrow \mathbf{Cat}$ by $P \colon \mathbf{Cat} \longrightarrow \mathbf{Free}$. For each category $C$, the category $PC$ is free on the underlying reflexive graph of $C$, and the counit functor $\varepsilon_C \colon PC \longrightarrow C$ is the identity on objects and sends a path of non-identity morphisms $(f_n,\ldots,f_1)$ in $C$ to its composite $f_n \circ \cdots \circ f_1$; note that $\varepsilon_C$ is bijective on objects and full. If $C$ is a free category, the unit morphism $\eta_C \colon C \longrightarrow PC$ is the identity on objects and sends a morphism $f$ with atomic decomposition $f = f_n \circ\cdots\circ f_1$ to the path $(f_n,\ldots,f_1)$.

\subsection{Finitely presentable free categories} \label{fpfreecatsec}
It follows from the equivalence $\mathbf{Gph} \simeq \mathbf{Free}$ that a free category is finitely presentable in $\mathbf{Free}$ if and only if it has finitely many objects and finitely many atomic morphisms; this is furthermore equivalent to the free category's being finitely presentable in $\mathbf{Cat}$. To prove the nontrivial direction of this latter equivalence, observe that a free category that is finitely presentable in $\mathbf{Cat}$ is generated under composition by some finite subset of its morphisms, which must contain all the atomic morphisms since they admit no non-trivial factorisations.

\subsection{Injective and surjective morphisms} \label{wfsfree}
A morphism of free categories is \emph{injective} if it is injective on objects and faithful. These are the monomorphisms in the category $\mathbf{Free}$; since $\mathbf{Free}$ is equivalent to a presheaf category, the pushout-product of two injective morphisms is therefore injective. For example, for each free category $C$, the unit morphism $\eta_C \colon C \longrightarrow PC$ is injective. 

 A morphism of free categories $F \colon A \longrightarrow B$ is \emph{surjective} if it is surjective on objects and if, for each pair of objects $x,y \in A$ and each atomic or identity morphism $g \colon Fx \longrightarrow Fy$ in $B$, there exists an atomic or identity morphism $f \colon x \longrightarrow y$ in $A$ such that $Ff = g$. Note that any surjective morphism of free categories is in particular full. The injective and surjective morphisms form the left class and right class respectively of a weak factorisation system on $\mathbf{Free}$, cofibrantly generated by the pair of injections $\emptyset \longrightarrow \mathbb{D}^0$ and $\partial\mathbb{D}^1 \longrightarrow \mathbb{D}^1$, where $\partial\mathbb{D}^1 = \mathbb{D}^0 + \mathbb{D}^0$.

This same pair of functors cofibrantly generates a weak factorisation system $(\mathscr{I},\mathscr{P})$ on $\mathbf{Cat}$, studied by Lack in \cite[\S4.3]{Lac02}, whose right class consists of those functors that are surjective on objects and full. A morphism of free categories belongs to the class $\mathscr{I}$ if and only if it is injective; the nontrivial direction of this equivalence follows  from \cite[Corollary 4.12]{Lac02}, which implies that every functor belonging to $\mathscr{I}$ is injective on objects and faithful.

\subsection{The funny symmetric monoidal closed structure} \label{funnysec}
In addition to its cartesian closed structure, the category $\mathbf{Free}$ admits another symmetric monoidal closed structure given by the restriction of the so-called funny symmetric monoidal structure on $\mathbf{Cat}$. The funny tensor product of two categories $A$ and $B$ is the category $A \mathbin{\square} B$ given by the following pushout in $\mathbf{Cat}$. 
\begin{equation*}
\cd{
\operatorname{ob}A \times \operatorname{ob}B \ar[r] \ar[d] & \operatorname{ob}A \times B \ar[d] \\
A \times \operatorname{ob}B \ar[r] & A \mathbin{\square} B \fatpushoutcorner
}
\end{equation*}
The unit object for this monoidal structure is the terminal category $\mathbb{D}^0$. 
To contrast it with the cartesian product in $\mathbf{Free}$, observe that the funny tensor product $\mathbb{D}^1 \mathbin{\square} \mathbb{D}^1$ is the free category on the following graph.
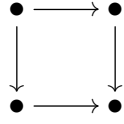
\begin{equation*}
\cd{
\bullet \ar[r] \ar[d] & \bullet \ar[d] \\
\bullet \ar[r] & \bullet
}
\end{equation*}

The internal hom for a pair of categories $A$ and $B$ in the funny closed structure on $\mathbf{Cat}$ is the category $[A,B]_{\mathbf{Cat}}$ of functors $A \longrightarrow B$ and unnatural transformations between them, where an unnatural transformation  $\theta \colon F \Longrightarrow G \colon A \longrightarrow B$ consists of a family of morphisms $\theta_x \colon Fx \longrightarrow Gx$ in $B$ indexed by the objects $x\in A$. If $A$ and $B$ are free, their internal hom in the funny closed structure on $\mathbf{Free}$ is the free category $[A,B]_{\mathbf{Free}}$ whose objects are the morphisms of free categories $A \longrightarrow B$, and whose atomic and identity morphisms are the unnatural transformations whose components are atomic or identity morphisms; note that there is a canonical functor  
 $[A,B]_{\mathbf{Free}} \longrightarrow [A,B]_{\mathbf{Cat}}$.

For each pair of free categories $A$ and $B$, the canonical morphism of free categories $(\pi_1,\pi_2) \colon A \mathbin{\square}B \longrightarrow A \boxtimes B$ fits into the following pushout square in $\mathbf{Free}$,
\begin{equation*}
\cd[@C=3.5em]{
\displaystyle\sum_{\substack{f,g \\ \text{atomic}}} \mathbb{D}^1 \mathbin{\square} \mathbb{D}^1 \ar[r]^-{\sum (\pi_1,\pi_2)} \save[]-<0pt,+10pt>{} \ar[d]^-{(f\mathbin{\square}g)} \restore & \displaystyle\sum_{\substack{f,g \\ \text{atomic}}} \mathbb{D}^1 \boxtimes \mathbb{D}^1 \save[]-<0pt,+10pt>{} \ar[d]^-{(f\boxtimes g)} \restore \\
A \mathbin{\square} B \ar[r]_-{(\pi_1,\pi_2)} & A \boxtimes B \fatpushoutcorner
}
\end{equation*}
where the coproducts are indexed by pairs of atomic morphisms $f$ in $A$ and $g$ in $B$.

\section{The category of flexible $2$-categories} \label{algcof}
In this section, we define the category of flexible $2$-categories and establish those of its categorical properties that we shall need in our construction of the left-induced model structure thereon.

\subsection{The category of flexible $\bm{2}$-categories}
\begin{definition}
A $2$-category is \emph{flexible} if its underlying category is free.
\end{definition}

\begin{remark}
This terminology comes from two-dimensional monad theory (cf.\ \cite[Remark 4.10]{Lac02}).
\end{remark}

\begin{definition}
A $2$-functor between flexible $2$-categories is a \emph{morphism of flexible $2$-categories} if its underlying functor is a morphism of free categories.
\end{definition}

This defines the category $\mathbf{Flex}$ of flexible $2$-categories as a non-full subcategory of $\bm{2}\mathbf{Cat}$. We shall continually use the observation that the category $\mathbf{Flex}$ and the inclusion functor $\mathbf{Flex} \longrightarrow \bm{2}\mathbf{Cat}$ fit into the following pullback square of categories,
\begin{equation} \label{pbsquare}
\cd{
\mathbf{Flex} \ar[r] \ar[d]_-U \fatpullbackcorner & \bm{2}\mathbf{Cat} \ar[d]^-U \\
\mathbf{Free} \ar[r]_-{} & \mathbf{Cat}
}
\end{equation}
which is moreover a bipullback square, since the forgetful functor $U \colon \bm{2}\mathbf{Cat} \longrightarrow \mathbf{Cat}$ is an isofibration (cf.\ \cite{pspb}). For instance, from the fact that the replete subcategory inclusion $\mathbf{Free} \longrightarrow \mathbf{Cat}$ is an isofibration, it follows that its pullback $\mathbf{Flex} \longrightarrow \bm{2}\mathbf{Cat}$ is also an isofibration, i.e.\ that $\mathbf{Flex}$ is a replete subcategory of $\mathbf{2Cat}$. 

We shall also make frequent use of the fact that the functor $U \colon \bm{2}\mathbf{Cat} \longrightarrow \mathbf{Cat}$ is a Grothendieck fibration whose cartesian morphisms are the locally fully faithful $2$-functors. Given a category $A$ and a $2$-category $B$, the $U$-cartesian lift of a functor $F \colon A \longrightarrow UB$ is the evident locally fully faithful $2$-functor $A' \longrightarrow B$ whose domain $A'$ is the $2$-category whose underlying category is $A$, in which a $2$-cell between a parallel pair of morphisms $f,g$ is a $2$-cell $Ff \Longrightarrow Fg$ in $B$, and whose $2$-cell composition operations and identities are induced from those of $B$.

\subsection{Colimits and limits in $\mathbf{Flex}$} \label{limsinflex}
The functor $U \colon \bm{2}\mathbf{Cat} \longrightarrow \mathbf{Cat}$ preserves all colimits and limits, since it has both a left and a right adjoint. It then follows from the bipullback square (\ref{pbsquare}) that $\mathbf{Flex}$ admits, and the functors $U \colon \mathbf{Flex} \longrightarrow \mathbf{Free}$ and $\mathbf{Flex} \longrightarrow \bm{2}\mathbf{Cat}$ preserve, all small colimits as well as coreflexive equalisers, since these are  preserved by the inclusion $\mathbf{Free} \longrightarrow \mathbf{Cat}$.

The product in $\mathbf{Flex}$ of a family $(A_i)_{i \in I}$ of flexible $2$-categories, which we write as $\bigboxtimes_{i \in I}A_i$ to distinguish it from the cartesian product in $\bm{2}\mathbf{Cat}$, can be constructed as the domain of the $U$-cartesian lift 
$$(\pi_i)_{i \in I} \colon \bigboxtimes_{i \in I} A_i \longrightarrow \prod_{i \in I} A_i$$
of the canonical functor 
$$(\pi_i)_{i \in I} \colon\bigboxtimes_{i \in I} U(A_i) \longrightarrow \prod_{i \in I} U(A_i) = U\left(\prod_{i\in I}A_i\right)$$ 
whose $i$\textsuperscript{th} component is the projection $\pi_i \colon \bigboxtimes_{i \in I} U(A_i) \longrightarrow U(A_i)$ for the cartesian product in $\mathbf{Free}$. 
Note that the forgetful functor $U \colon \mathbf{Flex} \longrightarrow \mathbf{Free}$ preserves products by construction. 
For example, the product $\mathbb{D}^1\mathbin{\boxtimes} \mathbb{D}^1$ in $\mathbf{Flex}$ is the flexible $2$-category presented by the following diagram.
\begin{equation*}
\cd{
\bullet \ar[r] \ar[d] \ar[dr] & \bullet \ar[d] \\
\bullet \ar[r] \ar@{}[ur]|(0.3){\cong} \ar@{}[ur]|(0.7){\cong}& \bullet
}
\end{equation*}

Since all small limits can be constructed from small products and coreflexive equalisers, we have proved the following proposition.

\begin{proposition}
The category $\mathbf{Flex}$ admits, and the forgetful functor $U \colon \mathbf{Flex} \longrightarrow \mathbf{Free}$ preserves, all small colimits and limits. The inclusion functor $\mathbf{Flex} \longrightarrow \bm{2}\mathbf{Cat}$ preserves all small colimits and coreflexive equalisers.
\end{proposition}

\subsection{The path $\bm{2}$-category functor} \label{qsec}
The inclusion $\mathbf{Flex} \longrightarrow \bm{2}\mathbf{Cat}$, being the pullback of a left adjoint functor along a Grothendieck fibration, is a left adjoint functor whose right adjoint $Q \colon \bm{2}\mathbf{Cat} \longrightarrow \mathbf{Flex}$ may be constructed as follows. For each $2$-category $A$, we define the flexible $2$-category $QA$, which we call the \emph{path $2$-category} of $A$, to be the domain of the $U$-cartesian lift $$\varepsilon_A \colon QA \longrightarrow A$$ of the counit functor $$\varepsilon_{UA} \colon PUA \longrightarrow UA.$$ These cartesian lifts $\varepsilon_A \colon QA \longrightarrow A$ are the components of the counit of the adjunction; note that they are bijective on objects, full, and locally fully faithful by construction.

Explicitly, $QA$ is the flexible $2$-category whose objects are the objects of $A$, whose $1$-cells are paths of non-identity morphisms in $A$, whose $1$-cell composition operation is concatenation of paths, whose $2$-cells $(f_m,\ldots,f_1) \Longrightarrow (g_n,\ldots,g_1)$ are named by $2$-cells $f_m\circ\cdots\circ f_1 \Longrightarrow g_n\circ\cdots\circ g_1$ in $A$, and whose $2$-cell composition operations and identities are induced from those of $A$. The counit $2$-functor $\varepsilon_A \colon QA \longrightarrow A$ is the identity on objects, sends a $1$-cell $(f_n,\ldots,f_1)$ in $QA$ to the composite $f_n\circ\cdots\circ f_1$ in $A$, and sends a $2$-cell in $QA$ to the $2$-cell in $A$ that names it.

If $A$ is flexible, the unit morphism $\eta_A \colon A \longrightarrow QA$ is the identity on objects, sends a morphism $f$ with atomic decomposition $f = f_n \circ\cdots\circ f_1$ to the path $(f_n,\ldots,f_1)$, and sends a $2$-cell $\alpha \colon f \Longrightarrow g$ to the $2$-cell $\eta_A(f) \Longrightarrow \eta_A(g)$ named by $\alpha$. Note that $\eta_A$ is bijective on objects, faithful, and locally fully faithful; it is also essentially full, since any path $(f_n,\ldots,f_1)$ is isomorphic to the path $\eta_A(f_n\circ\cdots\circ f_1)$ via the invertible $2$-cell in $QA$ between them named by the identity $2$-cell on $f_n\circ\cdots\circ f_1$ in $A$. 

For example, the unit morphism $\eta_{[2]} \colon [2] \longrightarrow Q[2]$, where $[2] = \{0 < 1 < 2\}$, is the inclusion displayed below.
\begin{equation*}
\left(\cd{
0 \ar[r]^-{01} & 1 \ar[r]^-{12} & 2
}\right)
\qquad
\xrightarrow{\hspace*{3em}}
\qquad
\left(\cd{
& 1 \ar[dr]^-{12} \ar@{}[d]|(0.6){\vcong} \\
0 \ar[ur]^-{01} \ar[rr]_{02} &{} & 2
}\right)
\end{equation*}

\begin{remark} \label{comonadremark}
It follows from the crude monadicity theorem that the inclusion functor $\mathbf{Flex} \longrightarrow \mathbf{2Cat}$ is strictly comonadic. 
Hence a $2$-category is flexible if and only if it admits a (necessarily unique) coalgebra structure for the induced comonad on $\mathbf{2Cat}$, which we call the path $2$-category comonad. This comonad satisfies the ``redundant coassociativity'' property of \cite{menni}, which is to say that a $2$-category $A$ is flexible if and only if the counit $2$-functor $\varepsilon_A \colon QA \longrightarrow A$ admits a (necessarily unique) section in $\mathbf{2Cat}$.
\end{remark}

\subsection{Local finite presentability of $\mathbf{Flex}$}
To prove that the cocomplete category $\mathbf{Flex}$ is locally finitely presentable, it suffices by \cite[Theorem 1.11]{arbook} to produce a strongly generating set of finitely presentable objects in $\mathbf{Flex}$. Our strongly generating set will consist of the terminal category $\mathbb{D}^0$, the free-living arrow $\mathbb{D}^1$, and the family of $2$-categories $\mathbb{D}^2[m,n]$ for $m,n \geq 0$, which we  now define.

In words, $\mathbb{D}^2[m,n]$ is the free $2$-category containing a $2$-cell whose source is the composite of a path of length $m$ and whose target is the composite of a path of length $n$. In a picture, $\mathbb{D}^2[m,n]$ is the free $2$-category on the following computad.
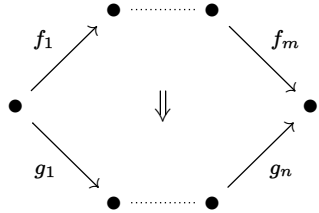
\begin{equation*}
\cd{
& \bullet \ar@{..}[r]  \dtwocell{ddr}{} & \bullet \ar[dr]^-{f_m} & \\
\bullet \ar[ur]^-{f_1} \ar[dr]_-{g_1} && &\bullet \\
& \bullet \ar@{..}[r] & \bullet \ar[ur]_-{g_n} &
}
\end{equation*}
Note that $\mathbb{D}^2[1,1] = \mathbb{D}^2$, the free $2$-category containing a $2$-cell.

\begin{lemma} \label{sglemma}
The flexible $2$-categories $\mathbb{D}^0$, $\mathbb{D}^1$, and $\mathbb{D}^2[m,n]$ for $m,n \geq 0$ form a strongly generating set for $\mathbf{Flex}$.
\end{lemma}
\begin{proof}
Let $F \colon A \longrightarrow B$ be a morphism of flexible $2$-categories and suppose that the induced function $\mathbf{Flex}(X,A) \longrightarrow \mathbf{Flex}(X,B)$ is a bijection whenever $X$ is one of the $2$-categories listed in the statement of the lemma. Taking $X=\mathbb{D}^0$ and $X=\mathbb{D}^1$, we see that the underlying functor of $F$ is an isomorphism.

Now, let $f,g$ be a parallel pair of morphisms in $A$, with $f = f_m \circ \cdots \circ f_1$ and $g = g_n \circ \cdots \circ g_1$ their atomic decompositions, and
 let $\beta \colon Ff \Longrightarrow Fg$ be a $2$-cell in $B$. Then $Ff = Ff_m \circ \cdots \circ Ff_1$ and $Fg = Fg_n \circ \cdots \circ Fg_1$ are atomic decompositions in $B$, and so there exists a morphism of flexible $2$-categories $\mathbb{D}^2[m,n] \longrightarrow B$ sending the non-identity $2$-cell in $\mathbb{D}^2[m,n]$ to $\beta$. Taking $X = \mathbb{D}^2[m,n]$, and recalling that $F$ is an isomorphism on underlying categories, we see that there exists a unique $2$-cell $\alpha \colon f \Longrightarrow g$ such that $F\alpha = \beta$. Hence $F$ is an isomorphism of $2$-categories.
\end{proof}

\begin{lemma} \label{fplemma}
A flexible $2$-category is finitely presentable in $\mathbf{Flex}$ if it is finitely presentable in $\bm{2}\mathbf{Cat}$.
\end{lemma}
\begin{proof}
Let $A$ be a flexible $2$-category and suppose that it is finitely presentable in $\mathbf{2Cat}$. Then its underlying free category $UA$ is finitely presentable in $\mathbf{Cat}$ and hence also finitely presentable in $\mathbf{Free}$ (see \S\ref{fpfreecatsec}). By the pullback square (\ref{pbsquare}), we have the following pullback square in the functor category $\mathbf{Fun}(\mathbf{Flex},\mathbf{Set})$.
\begin{equation*}
\cd{
\mathbf{Flex}(A,-) \ar[r] \ar[d]_-U \fatterpullbackcorner & \mathbf{2Cat}(A,-) \ar[d]^-U \\
\mathbf{Free}(UA,U-) \ar[r] & \mathbf{Cat}(UA,U-)
}
\end{equation*}
Since the four functors in the pullback square (\ref{pbsquare}) preserve filtered colimits, and since $A \in \mathbf{2Cat}$, $UA \in \mathbf{Free}$, and $UA \in \mathbf{Cat}$ are finitely presentable, it follows from the commutativity of filtered colimits with finite limits in $\mathbf{Set}$ that the functor $\mathbf{Flex}(A,-) \colon \mathbf{Flex} \longrightarrow \mathbf{Set}$ preserves filtered colimits, and hence that $A$ is finitely presentable in $\mathbf{Flex}$.
\end{proof}

\begin{remark}
The converse of Lemma \ref{fplemma} is also true, as one may prove by showing that the path $2$-category functor $Q \colon \mathbf{2Cat} \longrightarrow \mathbf{Flex}$ preserves filtered colimits.
\end{remark}

Since the flexible $2$-categories $\mathbb{D}^0$, $\mathbb{D}^1$, and $\mathbb{D}^2[m,n]$ for $m,n \geq 0$ are manifestly finitely presentable in $\mathbf{2Cat}$, Lemmas \ref{sglemma} and \ref{fplemma} together imply the following proposition.

\begin{proposition} \label{lfpprop}
The category $\mathbf{Flex}$ is locally finitely presentable.
\end{proposition}

\section{The left-induced model structure} \label{modstr}
The purpose of this section is to prove that the category $\mathbf{Flex}$ admits a model structure left-induced from Lack's model structure for $2$-categories along the inclusion functor $\mathbf{Flex} \longrightarrow \bm{2}\mathbf{Cat}$.

\subsection{Left-induced model structures} \label{leftysec}
A model structure on a category $\mathcal{M}$ is said to be \emph{left-induced} from a model structure on a category $\mathcal{N}$ along a functor $L \colon \mathcal{M} \longrightarrow \mathcal{N}$ if a morphism $f$ in $\mathcal{M}$ is a weak equivalence or cofibration precisely when $Lf$ is a weak equivalence or cofibration respectively in $\mathcal{N}$. Note that a model structure is determined by its classes of weak equivalences and cofibrations.

Given a model category $\mathcal{N}$ with class of weak equivalences $\mathcal{W}$ and class of cofibrations $\mathcal{C}$, to prove that there exists a (necessarily unique) model structure on a category $\mathcal{M}$ left-induced from $\mathcal{N}$ along a functor $L\colon \mathcal{M} \longrightarrow \mathcal{N}$, it is necessary and sufficient to show that the classes $L^{-1}(\mathcal{C})$ and $L^{-1}(\mathcal{C}\cap\mathcal{W})$ form the left classes of weak factorisation systems on $\mathcal{M}$, and that any morphism in $\mathcal{M}$ having the right lifting property with respect to the class $L^{-1}(\mathcal{C})$ belongs to the class $L^{-1}(\mathcal{W})$; this latter condition is called the ``acyclicity condition'' (see \cite{gmr}).  By \cite[Remark 3.8]{cellular}, if $L\colon \mathcal{M} \longrightarrow \mathcal{N}$ is a left adjoint functor between locally presentable categories and $\mathcal{N}$ is a combinatorial model category, the existence of the two weak factorisation systems is guaranteed, and so it suffices to verify the acyclicity condition alone; it follows also that the left-induced model structure on $\mathcal{M}$, if it exists, is  combinatorial.

\subsection{Lack's model structure for $\bm{2}$-categories} \label{lacksec}
In \cite{Lac02,Lac04}, Lack constructed a combinatorial model structure on the category $\mathbf{2Cat}$, which consists of the following classes of morphisms. A $2$-functor is a weak equivalence if and only if it is a \emph{biequivalence}, i.e.\ it is surjective on objects up to equivalence and an equivalence on hom-categories. A $2$-functor is a \emph{cofibration} if its underlying functor belongs to the class $\mathscr{I}$ (see \S\ref{wfsfree}) \cite[Proposition 4.12]{Lac02}; in particular, the cofibrant objects are precisely the flexible $2$-categories \cite[Theorem 4.8]{Lac02}. A $2$-functor $F \colon A \longrightarrow B$ is a fibration in Lack's model structure if and only if it is an \emph{equifibration}, i.e.\ it has the equivalence-lifting property (which states that for each object $a \in A$ and each equivalence $g \colon Fa \longrightarrow b$ in $B$, there exists an equivalence $f \colon a \longrightarrow a'$ in $A$ such that $Ff = g$) and is an isofibration on hom-categories; hence every $2$-category is fibrant in Lack's model structure.

The goal of this section is therefore to prove that there exists a model structure on $\mathbf{Flex}$ whose weak equivalences and cofibrations are those biequivalences and cofibrations respectively in $\mathbf{2Cat}$ that belong to the subcategory $\mathbf{Flex}$. Since the inclusion $\mathbf{Flex} \longrightarrow \mathbf{2Cat}$ is a left adjoint functor between locally finitely presentable categories by \S\ref{qsec} and Proposition \ref{lfpprop}, it suffices by \S\ref{leftysec} to check the acyclicity condition, which in our case amounts to the statement that 
any morphism of flexible $2$-categories having the right lifting property in $\mathbf{Flex}$ with respect to all cofibrations is a biequivalence.

\subsection{The acyclicity condition}
We shall verify the acyclicity condition by giving an explicit description of the weak factorisation system on $\mathbf{Flex}$ whose left class consists of the cofibrations in $\mathbf{Flex}$, and then observing that every morphism in the right class is a biequivalence.

\begin{proposition} \label{cofibprop}
A morphism of flexible $2$-categories is a cofibration if and only if it is injective on objects and faithful.
\end{proposition}
\begin{proof}
By definition, a $2$-functor is a cofibration if and only if its underlying functor belongs to the class of functors $\mathscr{I}$. But we saw in \S\ref{wfsfree} that a morphism of free categories belongs to $\mathscr{I}$ if and only if it is injective, i.e.\ injective on objects and faithful.
\end{proof}

We may therefore construct the weak factorisation system in question by an application of the following lemma.
\begin{lemma} \label{wfslemma}
Let $P \colon \mathcal{E} \longrightarrow \mathcal{B}$ be a Grothendieck fibration and let $(\mathcal{L},\mathcal{R})$ be a weak factorisation system on $\mathcal{B}$. There exists a weak factorisation system on $\mathcal{E}$ whose left class is the class $P^{-1}(\mathcal{L})$ and whose right class is the intersection of the class $P^{-1}(\mathcal{R})$ with the class of $P$-cartesian morphisms.
\end{lemma}
\begin{proof}
To factorise a morphism $f \colon X \longrightarrow Y$ in $\mathcal{E}$ as $f = h \circ g$, where $g$ belongs to the left class and $h$ belongs to the right class, first factorise $Pf$ as $r\circ l$, where $l \in \mathcal{L}$ and $r \in \mathcal{R}$. We may take $h \colon Z \longrightarrow Y$ to be the $P$-cartesian lift  of $r$, and then take $g \colon X \longrightarrow Z$ to be the unique morphism such that $Pg = l$ and $f = h\circ g$.  

It is immediate that both classes are closed under retracts. It remains to show, for each morphism $f \colon A \longrightarrow B$ in the left class and each morphism $g \colon C \longrightarrow D$ in the right class, that $f$ and $g$ are weakly orthogonal, i.e.\ that their  pullback-hom
$$f \mathbin{\widehat{\pitchfork}} g \colon \mathcal{E}(B,C) \longrightarrow \mathcal{E}(A,C) \times_{\mathcal{E}(A,D)} \mathcal{E}(B,D)$$
is surjective. But one can show that, since $g$ is $P$-cartesian, this pullback-hom is a pullback of the pullback-hom $Pf \mathbin{\widehat{\pitchfork}} Pg$, which is surjective since $Pf \in \mathcal{L}$ and $Pg \in \mathcal{R}$ are weakly orthogonal. Therefore $f \mathbin{\widehat{\pitchfork}} g$ is surjective.
\end{proof}

\begin{proposition} \label{wfsprop}
There exists a weak factorisation system on $\mathbf{Flex}$ whose left class consists of the cofibrations in $\mathbf{Flex}$, and to whose right class a morphism of flexible $2$-categories belongs if and only if its underlying functor is a surjection of free categories and it is locally fully faithful.
\end{proposition}
\begin{proof}
Note that the forgetful functor $U \colon \mathbf{Flex} \longrightarrow \mathbf{Free}$ is a Grothendieck fibration by the pullback square (\ref{pbsquare}), and that a morphism of flexible $2$-categories is $U$-cartesian  if and only if it is locally fully faithful. Hence the result follows from Proposition \ref{cofibprop} and an application of Lemma \ref{wfslemma} to the (injective, surjective) weak factorisation system  on $\mathbf{Free}$.
\end{proof}

\begin{remark}
The weak factorisation system of Proposition \ref{wfsprop} is cofibrantly generated by the boundary inclusions $\emptyset \longrightarrow \mathbb{D}^0$, $\partial\mathbb{D}^1 \longrightarrow \mathbb{D}^1$, and $\partial\mathbb{D}^2[m,n] \longrightarrow \mathbb{D}^2[m,n]$ (i.e.\ the inclusion of the underlying category) for $m,n\geq 0$, and the boundary-preserving projections displayed below for $m,n \geq 0$.
\begin{equation*}
\left(\cd{
& \bullet \ar@{..}[r]  \dtwocell{dd}{} & \dtwocell{dd}{} \bullet \ar[dr]^-{f_m} & \\
\bullet \ar[ur]^-{f_1} \ar[dr]_-{g_1} && &\bullet \\
& \bullet \ar@{..}[r] & \bullet \ar[ur]_-{g_n} &
}\right)
\qquad
\xrightarrow{\hspace*{3em}}
\qquad
\left(\cd{
& \bullet \ar@{..}[r]  \dtwocell{ddr}{} & \bullet \ar[dr]^-{f_m} & \\
\bullet \ar[ur]^-{f_1} \ar[dr]_-{g_1} && &\bullet \\
& \bullet \ar@{..}[r] & \bullet \ar[ur]_-{g_n} &
}\right)
\end{equation*}
\end{remark}

We may now deduce that the acyclicity condition holds.
\begin{proposition}
Every morphism of flexible $2$-categories having the right lifting property in $\mathbf{Flex}$ with respect to all cofibrations is a biequivalence.
\end{proposition}
\begin{proof}
Proposition \ref{wfsprop} implies that a morphism of flexible $2$-categories has the right lifting property in $\mathbf{Flex}$ with respect to all cofibrations if and only if its underlying functor is a surjection of free categories and it is locally fully faithful. But every such morphism is a biequivalence, since every surjection of free categories is surjective on objects and full.
\end{proof}

We have therefore proven the existence of the left-induced model structure on $\mathbf{Flex}$.

\begin{theorem}
There exists a combinatorial model structure on the category $\mathbf{Flex}$ in which a morphism of flexible $2$-categories is a weak equivalence if and only if it is a biequivalence, and  a cofibration if and only if it is injective on objects and faithful. Every object in this model category is cofibrant.
\end{theorem}

It is immediate that the left-induced model structure on $\mathbf{Flex}$ is Quillen equivalent to Lack's model structure on $\mathbf{2Cat}$. 
\begin{theorem} \label{quillenthm}
The adjunction
\begin{equation*}
\xymatrix{
\mathbf{Flex} \ar@<1.5ex>[rr]_-{\hdash}&& \ar@<1.5ex>[ll]^-{Q} \mathbf{2Cat}
}
\end{equation*}
is a Quillen equivalence between the left-induced model structure on $\mathbf{Flex}$ and Lack's model structure on $\mathbf{2Cat}$.
\end{theorem}
\begin{proof}
Since the left adjoint preserves weak equivalences and cofibrations by construction, the adjunction is a Quillen adjunction. It is moreover a Quillen equivalence because, for each flexible $2$-category $A$ and each $2$-category $B$, the unit and counit $2$-functors $\eta_A \colon A \longrightarrow QA$ and $\varepsilon_B \colon QB \longrightarrow B$ are biequivalences by \S\ref{qsec}.
\end{proof}

\begin{remark} \label{algcofrmk}
It was observed in Remark \ref{comonadremark} that $\mathbf{Flex}$ is isomorphic to the category of coalgebras for the path $2$-category comonad on $\mathbf{2Cat}$. Since, for every $2$-category $A$, the path $2$-category $QA$ is flexible and the counit $2$-functor $\varepsilon_A \colon QA \longrightarrow A$ is a biequivalence, this comonad is a cofibrant replacement comonad for Lack's model structure on $\mathbf{2Cat}$, and $\mathbf{Flex}$ is therefore a category of algebraically cofibrant objects for this model category.

Note that Garner's algebraic small object argument \cite{smallobject} applied to the set of generating cofibrations for Lack's model structure given in \cite[\S3]{Lac02} yields a different cofibrant replacement comonad, namely the pseudofunctor classifier comonad on $\mathbf{2Cat}$ (whereas the path $2$-category comonad $Q$ is the \emph{normal} pseudofunctor classifier comonad). The category of coalgebras for this comonad is isomorphic to the subcategory $\mathbf{Flex'}$ of $\mathbf{Flex}$ containing all flexible $2$-categories and those $2$-functors that preserve atomic morphisms. 

While this subcategory $\mathbf{Flex'}$ has many properties in common with $\mathbf{Flex}$,  it does not admit a model structure left-induced from Lack's model structure along the inclusion $\mathbf{Flex'} \longrightarrow \mathbf{2Cat}$. This is because the acyclicity condition fails to hold: the morphism $\mathbb{D}^0 + \mathbb{D}^0 \longrightarrow \mathbb{D}^0$ has the right lifting property in $\mathbf{Flex'}$ with respect to all cofibrations, but it is not a biequivalence. This disproves a conjecture of Ching and Riehl \cite[p.\ 172]{chingriehl}. (See \cite[Example 5.5]{bourkehenry} for a simpler variant of this counterexample.)
\end{remark}

\section{Fibrant objects and bicategories} \label{fibob}
In this section, we characterise the fibrant objects in the left-induced model structure on $\mathbf{Flex}$, and prove that the full subcategory of fibrant objects is equivalent to the category of bicategories and normal pseudofunctors. 

\subsection{Fibrant flexible $2$-categories}
We give several characterisations of the fibrant objects in the left-induced model structure on $\mathbf{Flex}$ in the following theorem.

\begin{theorem} \label{fibprop}
Let $A$ be a flexible $2$-category. The following are equivalent:
\begin{enumerate}[\normalfont(i)]
\item $A$ is a fibrant object in the left-induced model structure on $\mathbf{Flex}$;
\item the unit morphism $\eta_A \colon A \longrightarrow QA$ admits a retraction in $\mathbf{Flex}$;
\item $A$ is a retract of $QB$ in $\mathbf{Flex}$ for some $2$-category $B$;
\item $A$ has the right lifting property in $\mathbf{Flex}$ with respect to $\eta_{[2]} \colon [2] \longrightarrow Q[2]$;
\item every morphism in $A$ is isomorphic (via an invertible $2$-cell) to an atomic or identity morphism.
\end{enumerate}
\end{theorem}
\begin{proof}
(i)$\implies$(ii). The unit morphism $\eta_A \colon A \longrightarrow QA$ is bijective on objects, faithful, essentially full, and locally fully faithful, and hence a trivial cofibration in the left-induced model structure on $\mathbf{Flex}$. Since $A$ is a fibrant object in this model structure, it has the right lifting property in $\mathbf{Flex}$ with respect to $\eta_A$, whence $\eta_A$ admits a retraction in $\mathbf{Flex}$.

(ii)$\implies$(iii). By assumption, $A$ is a retract of $QA$ in $\mathbf{Flex}$.

(iii)$\implies$(i). Since the $2$-category $B$ is fibrant in Lack's model structure on $\mathbf{2Cat}$, it follows from the Quillen adjunction of Theorem \ref{quillenthm} that the flexible $2$-category $QB$ is a fibrant object in the left-induced model structure on $\mathbf{Flex}$. The result now follows from that the fact that the class of fibrant objects in a model category is closed under retracts.

(i)$\implies$(iv). The unit morphism $\eta_{[2]} \colon [2] \longrightarrow Q[2]$, displayed in \S\ref{qsec}, is a trivial cofibration in the left-induced model structure on $\mathbf{Flex}$, with respect to which therefore the fibrant object $A$ has the right lifting property in $\mathbf{Flex}$.

(iv)$\implies$(v). We prove by induction on $n \geq 0$ that every morphism in $A$ with atomic decomposition of length $n$ is isomorphic to an atomic or identity morphism. For $n=0,1$ this is immediate. Now suppose that $n > 1$ and let $f$ be a morphism with atomic decomposition $f = f_n \circ \cdots \circ f_1$. By the induction hypothesis, $g = f_{n-1} \circ \cdots \circ f_1$ is isomorphic to an atomic or identity morphism $h$, whence $f = f_n \circ g$ is isomorphic to $f_n \circ h$. The morphism $[2] \longrightarrow A$ that picks out the composable pair $(f_n,h)$ extends by assumption to a morphism $Q[2] \longrightarrow A$, which picks out an atomic or identity morphism $k$ and an isomorphism $f = f_n \circ h \cong k$.

(v)$\implies$(ii). 
By assumption, we may choose, for each morphism $f$ in $A$, a parallel atomic or identity morphism $r(f)$ and an invertible $2$-cell $\iota_f \colon f \cong r(f)$, being careful to choose $r(f) = f$ and $\iota_f = 1_f$ if $f$ is atomic or an identity.
We define the $2$-functor $R \colon QA \longrightarrow A$ to be the identity on objects, to send a morphism $(f_n,\ldots,f_1)$ in $QA$ to the composite $r(f_n) \circ \cdots \circ r(f_1)$, and to send a $2$-cell $\alpha \colon (f_m,\ldots,f_1) \Longrightarrow (g_n,\ldots,g_1)$ to the following vertical composite. 
\begin{equation*}
{
\xymatrix@C=1.7em{
r(f_m) \circ \cdots \circ r(f_1) \ar@{=>}[rr]^-{\iota_{f_m}^{-1} \circ \cdots \circ \iota_{f_1}^{-1}} && f_m \circ \cdots \circ f_1 \ar@{=>}[r]^-{\alpha} & g_n \circ \cdots \circ g_1 \ar@{=>}[rr]^-{\iota_{g_n} \circ \cdots \circ \iota_{g_1}} && r(g_n) \circ \cdots \circ r(g_1)
}}
\end{equation*}
It is straightforward to check that $R \colon QA \longrightarrow A$ is a morphism of flexible $2$-categories  and that it is a retraction of $\eta_A$.
\end{proof}

We shall refer to the full subcategory of $\mathbf{Flex}$ spanned by the fibrant objects in the left-induced model structure as the category of fibrant flexible $2$-categories. This subcategory admits the following two refinements of Lemma \ref{sglemma}, which we shall use later in this section.

\begin{lemma} \label{genlemma}
The $2$-categories $\mathbb{D}^0$, $\mathbb{D}^1$, $\mathbb{D}^2$, and $Q[2]$ form a generator for the category of fibrant flexible $2$-categories.
\end{lemma}
\begin{proof}
Let $F,G \colon A \longrightarrow B$ be morphisms of flexible $2$-categories and suppose that $A$ is fibrant. Suppose that $FH = GH$ for every morphism of flexible $2$-categories $H \colon X \longrightarrow A$ whose domain $X$ is one of the $2$-categories listed in the statement of the lemma. Taking $X = \mathbb{D}^0$ and $X = \mathbb{D}^1$, we see that the underlying functors of $F$ and $G$ are equal.

We now prove that $F$ and $G$ agree on $2$-cells. Since $A$ is fibrant, we may argue as in the proof of the implication (iv)$\implies$(v) of Theorem \ref{fibprop} to show that, for each morphism $f$ in $A$, there exists an atomic or identity morphism $r(f)$ and an invertible $2$-cell $\iota_f \colon f \cong r(f)$ in $A$ which is a pasting composite of invertible $2$-cells whose sources have atomic decompositions of length at most two and whose targets are atomic or identity morphisms. Taking $X = Q[2]$, we see that $F\iota_f = G\iota_f$ for every morphism $f$ in $A$. Furthermore, every $2$-cell $\alpha \colon f \Longrightarrow g$ in $A$ can be written as a vertical composite 
\begin{equation*}
\xymatrix@C=3em{
f \ar@{=>}[r]^-{\iota_f} & r(f) \ar@{=>}[r]^-{\iota_g \cdot \alpha \cdot \iota_f^{-1}} & r(g) \ar@{=>}[r]^-{\iota_g^{-1}} & g
}
\end{equation*}
whose middle factor is a $2$-cell whose source and target are atomic or identity morphisms. Hence, taking $X = \mathbb{D}^2$, we see that $F$ and $G$ agree on all $2$-cells, whence they are equal.
\end{proof}

\begin{lemma} \label{stronggenlemma}
The $2$-categories $\mathbb{D}^0$, $\mathbb{D}^1$, and $\mathbb{D}^2$ form a strong generator for the category of fibrant flexible $2$-categories. 
\end{lemma}
\begin{proof}
Let $F \colon A \longrightarrow B$ be a morphism of flexible $2$-categories and suppose that $A$ is fibrant. Suppose that the induced function $\mathbf{Flex}(X,A) \longrightarrow \mathbf{Flex}(X,B)$ is a bijection whenever $X$ is one of the $2$-categories listed in the statement of the lemma. Taking $X = \mathbb{D}^0$ and $X = \mathbb{D}^1$, we see that the underlying functor of $F$ is an isomorphism of categories.

Now, let $f,g$ be a parallel pair of morphisms in $A$. Since $A$ is fibrant, Theorem \ref{fibprop} implies that there exist atomic or identity morphisms $r(f)$ and $r(g)$ and invertible $2$-cells $\iota_f \colon f \cong r(f)$ and $\iota_g \colon g \cong r(g)$. We then have a commutative square of sets of $2$-cells
\begin{equation*}
\cd{
A(f,g) \ar[r]^-F \ar[d]_-{A(\iota_f^{-1},\iota_g)} \ar@{}[dr]|=& B(Ff,Fg) \ar[d]^-{B(F\iota_f^{-1},F\iota_g)} \\
A(r(f),r(g)) \ar[r]_-F & B(Fr(f),Fr(g))
}
\end{equation*}
in which the vertical morphisms are bijective. Taking $X = \mathbb{D}^2$, we see that the bottom morphism is also bijective, whence the top morphism is bijective. Hence $F$ is locally fully faithful, and therefore an isomorphism of $2$-categories.
\end{proof}

\begin{remark}
Beware that the strong generator of Lemma \ref{stronggenlemma} is not a generator. This apparent paradox is related to the fact that the category of fibrant flexible $2$-categories lacks equalisers (cf.\ \cite[Example 4.5]{Lac02}).
\end{remark}

\subsection{Bicategories as fibrant objects}
Our next goal is to prove that the category of fibrant flexible $2$-categories is equivalent to the category $\mathbf{Bicat}$ of bicategories and normal pseudofunctors. We shall do this by extending the path $2$-category functor $Q \colon \mathbf{2Cat} \longrightarrow \mathbf{Flex}$ to a fully faithful functor $Q \colon \mathbf{Bicat} \longrightarrow \mathbf{Flex}$ whose essential image consists of the fibrant flexible $2$-categories. This construction is a variant of the strictification of a bicategory (see \cite[\S4.10]{GPS}), and depends on the coherence theorems for bicategories and pseudofunctors (see \cite[\S2]{Gur13}).

First, we fix $n$-ary composition operations $(f_n,\ldots,f_1) \longmapsto \varepsilon(f_n,\ldots,f_1)$ in a bicategory, which we define inductively on $n \geq 0$ as follows. We define the composite of an empty path on an object to be the identity morphism on that object, the composite of a path $(f)$ of length one to be $f$, and for $n > 1$ we define $\varepsilon(f_n,\ldots,f_1) = \varepsilon(f_n,\ldots,f_2)\circ f_1$.  

Note that we shall occasionally denote paths in a bicategory by bold letters $\bm{f}$ and concatenation of paths by $\bm{g}\circ\bm{f}$. For each composable pair of paths $\bm{g}$ and $\bm{f}$ in a bicategory, there is a canonical generalised associativity isomorphism $\gamma \colon \varepsilon(\bm{g})\circ\varepsilon(\bm{f}) \cong \varepsilon(\bm{g}\circ\bm{f})$ by the coherence theorem for bicategories.

We now define the \emph{path $2$-category} of a bicategory $B$ to be the flexible $2$-category $QB$ whose underlying category is the free category on the underlying reflexive graph of $B$, whose $2$-cells $\bm{f} \Longrightarrow \bm{g}$ are named by $2$-cells $\varepsilon(\bm{f}) \Longrightarrow \varepsilon(\bm{g})$ in $B$, whose identity $2$-cells and vertical composition of $2$-cells are induced from those in $B$, and in which the horizontal composite of a horizontally composable pair of $2$-cells $\beta \colon \bm{h} \Longrightarrow \bm{k}$ and $\alpha \colon \bm{f} \Longrightarrow \bm{g}$ is named by the following composite in $B$.
\begin{equation*}
\xymatrix{
\varepsilon(\bm{h}\circ\bm{f})  \ar@{=>}[r]^-{\gamma^{-1}} & \varepsilon(\bm{h})\circ \varepsilon(\bm{f}) \ar@{=>}[r]^-{\beta\circ\alpha} & \varepsilon(\bm{k})\circ\varepsilon(\bm{g}) \ar@{=>}[r]^-{\gamma} & \varepsilon(\bm{k}\circ\bm{g})
}
\end{equation*}
Associativity of horizontal composition in the $2$-category $QB$ follows from the coherence theorem for bicategories.

For each bicategory $B$, we define a pair of normal pseudofunctors $\upsilon_B \colon B \longrightarrow QB$ and $\varepsilon_B \colon QB \longrightarrow B$ as follows. We define $\upsilon_B \colon B \longrightarrow QB$ to be the identity on objects, to send  a non-identity morphism $f$ in $B$ to the path $(f)$, to send a $2$-cell $\alpha \colon f \Longrightarrow g$ in $B$ to the $2$-cell $\upsilon_B(f) \Longrightarrow \upsilon_B(g)$ in $QB$ named by $\alpha$, and to have coherence constraints in $QB$ named by identity $2$-cells in $B$. We define $\varepsilon_B \colon QB \longrightarrow B$ to be the identity on objects, to send a path $(f_n,\ldots,f_1)$ to its composite $\varepsilon(f_n,\ldots,f_1)$, to send a $2$-cell in $QB$ to the $2$-cell in $B$ that names it, and to have as coherence constraints the generalised associativity isomorphisms $\gamma \colon \varepsilon(\bm{g})\circ\varepsilon(\bm{f}) \cong \varepsilon(\bm{g}\circ\bm{f})$ of $B$. The proof that $\varepsilon_B$ defines a normal pseudofunctor depends on the coherence theorem for bicategories.

It is immediate from the definitions of these normal pseudofunctors that $\varepsilon_B \circ \upsilon_B = 1_B$ for every bicategory $B$. This proves the following useful observation.

\begin{proposition} \label{retractlemma}
Every bicategory is a retract in $\mathbf{Bicat}$ of a $2$-category.
\end{proposition}

The path $2$-category construction defines a functor $Q \colon \mathbf{Bicat} \longrightarrow \mathbf{Flex}$ which extends the path $2$-category functor $Q \colon \mathbf{2Cat} \longrightarrow \mathbf{Flex}$. For each normal pseudofunctor between bicategories $F \colon B\longrightarrow C$, the induced morphism of flexible $2$-categories $QF \colon QB \longrightarrow QC$  is given on objects and atomic morphisms by the action of $F$ on objects and morphisms, and  sends a $2$-cell $\alpha \colon \bm{f} \Longrightarrow \bm{g}$ in $QB$ to the $2$-cell $QF(\bm{f}) \Longrightarrow QF(\bm{g})$ in $QC$ named by the composite
\begin{equation*}
\xymatrix{
\varepsilon(F\bm{f}) \ar@{=>}[r]^-{\varphi} & F(\varepsilon(\bm{f})) \ar@{=>}[r]^-{F\alpha}& F(\varepsilon(\bm{g})) \ar@{=>}[r]^-{\varphi^{-1}} & \varepsilon(F\bm{g}),
}
\end{equation*}
where  $\varphi$ denotes the generalised coherence constraints of the normal pseudofunctor $F$. These constraints are  determined by the coherence theorem for pseudofunctors, which may also be used to prove that $QF$ preserves horizontal composition of $2$-cells. 

The proof that the path $2$-category functor $Q \colon \mathbf{Bicat} \longrightarrow \mathbf{Flex}$ preserves composition is facilitated by Lemma \ref{genlemma} and the following proposition.

\begin{proposition} \label{bicattofibprop}
For every bicategory $B$, the path $2$-category $QB$ is a fibrant flexible $2$-category.
\end{proposition}
\begin{proof}
Let $(f_n,\ldots,f_1)$ be a path of non-identity morphisms in $B$. There exists an invertible $2$-cell $(f_n,\ldots,f_1) \cong \upsilon_B(\varepsilon(f_n,\ldots,f_1))$ in $QB$ named by the identity $2$-cell on $\varepsilon(f_n,\ldots,f_1)$. Thus every morphism in $QB$ is isomorphic to an atomic or identity morphism, and the result follows from the implication (v)$\implies$(i) of Theorem \ref{fibprop}.
\end{proof}

We can now state that the normal pseudofunctors $\upsilon_B \colon B \longrightarrow QB$ form the components of a natural transformation from the identity functor on $\mathbf{Bicat}$ to the composite functor $Q \colon \mathbf{Bicat} \longrightarrow \mathbf{Bicat}$. We have already seen in \S\ref{qsec} that the $2$-functors $\varepsilon_A \colon QA \longrightarrow A$, for $A$ a $2$-category, form the components of a natural transformation from the composite functor $Q \colon \mathbf{2Cat} \longrightarrow \mathbf{2Cat}$ to the identity functor on $\mathbf{2Cat}$. It is straightforward to verify the remaining triangle identity $\varepsilon_{QB} \circ Q\upsilon_B = 1_{QB}$ for each bicategory $B$, which completes the proof of the following proposition.

\begin{proposition} \label{adjprop}
The path $2$-category functor $Q \colon \mathbf{Bicat} \longrightarrow \mathbf{2Cat}$ is left adjoint to the inclusion functor $\mathbf{2Cat} \longrightarrow \mathbf{Bicat}$, with unit morphisms $\upsilon_B \colon B \longrightarrow QB$ and counit morphisms $\varepsilon_A \colon QA \longrightarrow A$.
\end{proposition}

\begin{remark}
The path $2$-category of a bicategory is therefore what has been called elsewhere its normal strictification. It also follows that the path $2$-category of a $2$-category is none other than its normal pseudofunctor classifier.
\end{remark}

All is now prepared for giving a simple proof of the following proposition.

\begin{proposition} \label{ffprop}
The path $2$-category functor $Q \colon \mathbf{Bicat} \longrightarrow \mathbf{Flex}$ is fully faithful.
\end{proposition}
\begin{proof}
We are required to prove that, for each pair of bicategories $A$ and $B$, the function 
\begin{equation*}
\xymatrix{
\mathbf{Bicat}(B,A) \ar[r]^-Q & \mathbf{Flex}(QB,QA)
}
\end{equation*}
is a bijection. Since every bicategory is a retract of a $2$-category in $\mathbf{Bicat}$ by Proposition \ref{retractlemma}, it suffices to consider the case when $A$ is a $2$-category. In that case, the $2$-functor $\varepsilon_A \colon QA \longrightarrow A$ is a counit morphism for the adjunction between the inclusion functor $\mathbf{Flex} \longrightarrow \mathbf{2Cat}$ and the path $2$-category functor $Q \colon \mathbf{2Cat} \longrightarrow \mathbf{Flex}$. Hence the function
\begin{equation*}
\xymatrix@C=3em{
\mathbf{Flex}(QB,QA) \ar[r]^-{\varepsilon_A \circ (-)} & \mathbf{2Cat}(QB,A)
}
\end{equation*} 
is a bijection. So the result follows from the fact that the composite function
\begin{equation*}
\cd[@C=3em]{
\mathbf{Bicat}(B,A) \ar[r]^-{Q} & \mathbf{Flex}(QB,QA) \ar[r]^-{\varepsilon_A\circ (-)} & \mathbf{2Cat}(QB,A)
}
\end{equation*}
is also a bijection by Proposition \ref{adjprop}. 
\end{proof}

It remains to characterise the essential image of the functor $Q \colon \mathbf{Bicat} \longrightarrow \mathbf{Flex}$.

\begin{theorem} \label{bicatthm}
The path $2$-category functor $Q \colon \mathbf{Bicat} \longrightarrow \mathbf{Flex}$ determines an equivalence between the category of bicategories and normal pseudofunctors and the full subcategory of $\mathbf{Flex}$ spanned by the fibrant objects in the left-induced model structure.
\end{theorem}
\begin{proof}
We have proven in Propositions \ref{bicattofibprop} and \ref{ffprop} that the path $2$-category functor $Q \colon \mathbf{Bicat} \longrightarrow \mathbf{Flex}$ is fully faithful and that its image is contained in the full subcategory of $\mathbf{Flex}$ spanned by the fibrant flexible $2$-categories. It therefore remains to prove that every fibrant flexible $2$-category is isomorphic to the path $2$-category of a bicategory.

Let $A$ be a fibrant flexible $2$-category. We construct a bicategory $B$ as follows. First, by Theorem \ref{fibprop}, we may choose, for each composable pair of atomic or identity morphisms $g,f$ in $A$, an atomic or identity morphism $g * f$ and an invertible $2$-cell $\iota_{g,f} \colon g \circ f \cong g*f$ in $A$, such that whenever $f$ or $g$ is an identity morphism, $g*f = g\circ f$ and $\iota_{g,f} = 1_{g\circ f}$. We define the objects of $B$ to be the objects of $A$, the identity and non-identity morphisms of $B$ to be the identity and atomic morphisms of $A$ respectively, composition of morphisms to be given by $(g,f) \longmapsto g*f$, the $2$-cells $f \Longrightarrow g$ in $B$ to be the $2$-cells $f \Longrightarrow g$ in $A$, the identity $2$-cells and vertical composition of $2$-cells to be defined as in $A$, the horizontal composition of a horizontally composable pair of $2$-cells $\beta \colon h \Longrightarrow k$ and $\alpha \colon f \Longrightarrow g$ in $B$ to be the composite $2$-cell
\begin{equation*}
\cd{
h * f \ar@{=>}[r]^-{\iota_{h,f}^{-1}} & h \circ f \ar@{=>}[r]^-{\beta \circ \alpha} & k \circ g \ar@{=>}[r]^-{\iota_{k,g}} & k * g
}
\end{equation*}
in $A$, the unit constraints to be identities, and the associativity constraint $\mathfrak{a}_{h,g,f}$ for a composable triple of morphisms $h,g,f$ in $B$ to be the following composite $2$-cell in $A$. 
\begin{equation*}
\cd[@C=3em]{
(h * g)*f \ar@{=>}[r]^-{\iota_{h*g,f}^{-1}} & (h*g) \circ f \ar@{=>}[r]^{\iota_{h,g}^{-1} \circ 1} & h\circ g\circ f \ar@{=>}[r]^-{1\circ \iota_{g,f}} & h \circ (g*f) \ar@{=>}[r]^-{\iota_{h,g*f}} & h*(g*f)
}
\end{equation*}
The naturality of the associativity constraints follows from the definition of horizontal composition, while the pentagon coherence axiom is proven in Figure \ref{figure}, in which we denote composition in $B$ by juxtaposition, and whose outer regions commute by definition and whose interior diamond commutes by functoriality of horizontal composition in $A$.
\begin{figure}
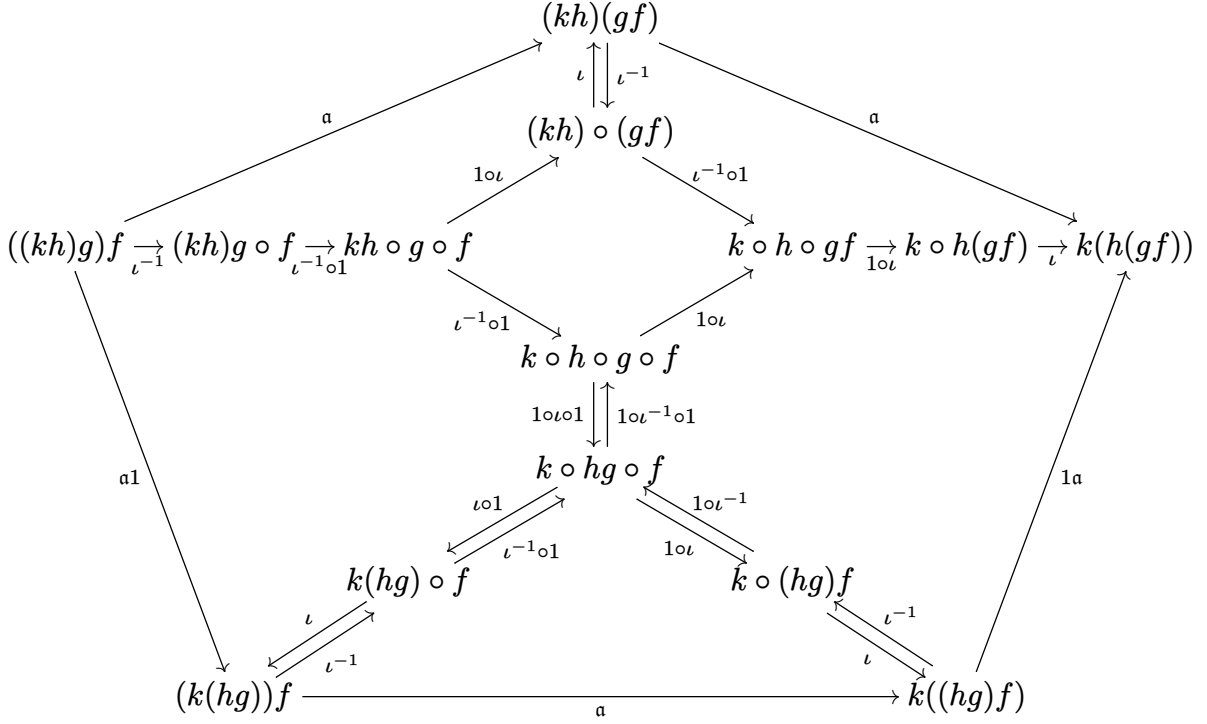

\begin{equation*}
\cd[@C=.9em]{
&&& (kh)(gf) \ar[ddrrr]^-{\mathfrak{a}} \ar@<.5ex>[d]^-{\iota^{-1}}\\
&&& (kh)\circ(gf) \ar[dr]^-{\iota^{-1} \circ1} \ar@<.5ex>[u]^-{\iota} \\
((kh)g)f \ar[uurrr]^-{\mathfrak{a}} \ar[r]_-{\iota^{-1}} \ar[ddddr]_-{\mathfrak{a}  1} & (kh)g\circ f \ar[r]_-{\iota^{-1}\circ 1} & kh\circ g\circ  f \ar[ur]^-{1\circ  \iota} \ar[dr]_-{\iota^{-1}\circ 1} && k\circ h\circ  gf \ar[r]_-{1\circ  \iota} & k\circ  h(gf) \ar[r]_-{\iota} & k(h(gf)) \\
&&& k\circ  h\circ  g\circ  f \ar[ur]_-{1\circ  \iota} \ar@<-.5ex>[d]_-{1\circ  \iota \circ 1}\\
&&& k\circ  hg\circ f \ar@<-.5ex>[u]_-{1\circ  \iota^{-1}\circ  1} \ar@<-.5ex>[dl]_-{\iota\circ 1} \ar@<-.5ex>[dr]_{1\circ\iota}\\
&& k(hg)\circ f \ar@<-.5ex>[ur]_-{\iota^{-1}\circ 1} \ar@<-.5ex>[dl]_-{\iota} && k\circ  (hg)f\ar@<-.5ex>[ul]_-{1\circ\iota^{-1}} \ar@<-.5ex>[dr]_-{\iota} \\
& (k(hg))f \ar[rrrr]_-{\mathfrak{a}} \ar@<-.5ex>[ur]_-{\iota^{-1}} &&&& k((hg)f) \ar[uuuur]_-{1  \mathfrak{a}} \ar@<-.5ex>[ul]_-{\iota^{-1}}
}
\end{equation*}
\caption{Proof of the pentagon axiom for the bicategory extracted from a fibrant flexible $2$-category.}  \label{figure}
\end{figure}

There is a normal pseudofunctor $I \colon B \longrightarrow A$ given by the evident inclusion on objects, morphisms, and $2$-cells, and whose coherence constraints are the isomorphisms $\iota_{g,f}$. By Proposition \ref{adjprop}, this normal pseudofunctor corresponds to a $2$-functor $\hat{I} \colon QB \longrightarrow A$, which is in fact a morphism of flexible $2$-categories since $F$ sends every morphism in $B$ to an atomic or identity morphism in $A$. It now follows from Lemma \ref{stronggenlemma} that the morphism of flexible $2$-categories $\hat{I} \colon QB \longrightarrow A$ is an isomorphism, since the normal pseudofunctor $I$ is by definition bijective on objects, faithful, full on atomic and identity morphisms, and locally fully faithful.
\end{proof}
 
We may now add the following equivalent condition to Theorem \ref{fibprop}.

\begin{corollary}
A flexible $2$-category is fibrant in the left-induced model structure on $\mathbf{Flex}$ if and only if it is isomorphic to the path $2$-category of a bicategory.
\end{corollary}

We may also deduce the following universal property of the category $\mathbf{Bicat}$ (cf.\ \cite[\S6.3]{bourke}).

\begin{corollary} \label{retractcor}
The category of bicategories and normal pseudofunctors is the Cauchy completion of the category of $2$-categories and normal pseudofunctors. 
\end{corollary}
\begin{proof}
We showed in Proposition \ref{retractlemma} that every bicategory is a retract of a $2$-category in $\mathbf{Bicat}$. Furthermore, since $\mathbf{Bicat}$ is equivalent to the category of fibrant objects in a model category by Theorem \ref{bicatthm}, it is Cauchy complete.
\end{proof}

\section{Fibrations} \label{fibsec}
In this section, we characterise the fibrations with fibrant codomain in the left-induced model structure on $\mathbf{Flex}$. We approach this characterisation via something of a detour through a more general class of morphisms of flexible $2$-categories, analogous to the inner fibrations of simplicial sets which play a fundamental role in the homotopy theory of quasi-categories. 
\subsection{Inner fibrations}
\begin{definition} A morphism of flexible $2$-categories $F \colon A \longrightarrow B$ is an \emph{inner fibration} if for every morphism $f \colon a \longrightarrow b$ in $A$, every atomic or identity morphism $g \colon Fa \longrightarrow Fb$ and every invertible $2$-cell $\beta \colon Ff \cong g$ in $B$, there exists an atomic or identity morphism $h \colon a \longrightarrow b$ and an invertible $2$-cell $\alpha \colon f \cong h$ in $A$ such that $F\alpha = \beta$. 
\end{definition}

\begin{example} A flexible $2$-category $A$ is fibrant if and only if the unique morphism $A \longrightarrow \mathbb{D}^0$ is an inner fibration. This is a restatement of the equivalence (i)$\iff$(v) of Theorem \ref{fibprop}.
\end{example}

\begin{example} \label{secondexample}
A morphism of flexible $2$-categories is an inner fibration if and only if it has the right lifting property in $\mathbf{Flex}$ with respect to the source inclusions displayed below for all $n \geq 0$.
\begin{equation*}
\left(\cd{
0 \ar[r] & \ar@{..}[r] & \ar[r]& n
}\right)
\qquad
\xrightarrow{\hspace*{3em}}
\qquad
\left(\cd{
& \ar@{.}[r] \ar@{}[dr]|\vcong & \ar[dr] \\
0 \ar[ur] \ar[rrr] &&& n
}\right)
\end{equation*}
Since each of these morphisms is a trivial cofibration, every fibration in the left-induced model structure on $\mathbf{Flex}$ is in particular an inner fibration.
\end{example}

The following proposition shows that the right lifting property with respect to a single morphism suffices to characterise the inner fibrations with fibrant codomain. Note that we let $\mathbb{I}^2$ denote the free $2$-category containing an invertible $2$-cell, which is presented by the following diagram.
\begin{equation*}
\xymatrix{
\bullet \ar@/^.8pc/[r] \ar@/_.8pc/[r] \ar@{}[r]|{\vcong} & \bullet
}
\end{equation*}

\begin{proposition} \label{innerfibprop}
Let $F \colon A \longrightarrow B$ be a morphism of flexible $2$-categories, and suppose that $B$ is fibrant. The following are equivalent:
\begin{enumerate}[\normalfont(i)]
\item $F$ is an inner fibration;
\item $F$ has the right lifting property in $\mathbf{Flex}$ with respect to $\eta_{[2]} \colon [2] \longrightarrow Q[2]$;
\item $A$ is fibrant and $F$ has the right lifting property in $\mathbf{Flex}$ with respect to the source inclusion $\mathbb{D}^1 \longrightarrow \mathbb{I}^2$.
\end{enumerate}
\end{proposition}
\begin{proof}
(i)$\implies$(ii). A lifting problem
\begin{equation*}
\cd{
[2] \ar[r] \ar[d]_-{\eta_{[2]}} & A \ar[d]^-F \\
Q[2] \ar[r] \ar@{..>}[ur]& B
}
\end{equation*}
determines a morphism $f \colon a \longrightarrow b$ in $A$ with atomic decomposition of length at most two, an atomic or identity morphism $g \colon Fa \longrightarrow Fb$ in $B$, and an invertible $2$-cell $\beta \colon Ff \cong g$ in $B$. Since $F$ is an inner fibration, there exists an atomic or identity morphism $h \colon a \longrightarrow b$ and an invertible $2$-cell $\alpha \colon f \cong h$ in $A$ such that $F\alpha = \beta$, which determines a solution to the lifting problem.

(ii)$\implies$(iii). Since $B$ is fibrant, it follows from the implication (iv)$\implies$(i) of Theorem \ref{fibprop} that $A$ is also fibrant. Since the source inclusion $\mathbb{D}^1 \longrightarrow \mathbb{I}^2$ is a pushout of the morphism $\eta_{[2]} \colon [2] \longrightarrow Q[2]$ in $\mathbf{Flex}$, any morphism which has the right lifting property with respect to the latter also has the right lifting property with respect to the former.

(iii)$\implies$(i). Let $f \colon a \longrightarrow b$ be a morphism in $A$, $g \colon Fa \longrightarrow Fb$ an atomic or identity morphism in $B$, and $\beta \colon Ff \cong g$ an invertible $2$-cell in $B$. Since $A$ is fibrant, Theorem \ref{fibprop} implies that there exist an atomic or identity morphism $f' \colon a \longrightarrow b$ and an invertible $2$-cell $\alpha \colon f \cong f'$ in $A$. The morphism $f'$ and the composite invertible $2$-cell $\beta\cdot F\alpha^{-1} \colon Ff' \cong g$ determine a lifting problem
\begin{equation*}
\cd{
\mathbb{D}^1 \ar[r] \ar[d] & A \ar[d]^-F \\
\mathbb{I}^2 \ar[r] \ar@{..>}[ur] & B
}
\end{equation*}
whose solution determines an atomic or identity morphism $h \colon a \longrightarrow b$ and an invertible $2$-cell $\gamma \colon f' \cong h$ in $A$ such that  $F\gamma = \beta\cdot F\alpha^{-1}$, whence $F(\gamma \cdot \alpha) = \beta$. Hence $F$ is an inner fibration.
\end{proof}

It is possible to construct a weak factorisation system on $\mathbf{Flex}$ whose right class consists of the inner fibrations by applying the small object argument to the set of morphisms described in Example \ref{secondexample}. We shall make an alternative construction of this weak factorisation system, which has the advantage of giving the following explicit description of the left class.

\begin{definition}
A morphism of flexible $2$-categories is \emph{inner anodyne}, or an \emph{inner anodyne extension}, if it is bijective on objects, faithful, essentially full, and locally fully faithful.
\end{definition}
\begin{example}
A morphism of flexible $2$-categories is inner anodyne if and only if it is a bijective-on-objects trivial cofibration. In particular, the unit morphism $\eta_A \colon A \longrightarrow QA$ is inner anodyne for every flexible $2$-category $A$.
\end{example}

We shall produce a functorial factorisation of morphisms of flexible $2$-categories into inner anodyne extensions followed by inner fibrations by means of the following construction.

We define the \emph{strand $2$-category} of a morphism of flexible $2$-categories  $F \colon A \longrightarrow B$ to be the flexible $2$-category $\mathcal{S}(F)$ whose objects are the objects of $A$, in which an atomic or identity morphism $a \longrightarrow b$ is a triple $(f,g,\alpha)$, where $f \colon a \longrightarrow b$ is a morphism in $A$,  $g \colon Fa \longrightarrow Fb$ is an atomic or identity morphism in $B$, and $\alpha \colon Ff \cong g$ is an invertible $2$-cell in $B$, and where a $2$-cell $(f_m,g_m,\alpha_m)\circ\cdots\circ(f_1,g_1,\alpha_1) \Longrightarrow (h_n,k_n,\beta_n)\circ\cdots\circ(h_1,k_1,\beta_1)$ is a pair $(\gamma,\delta)$ where $\gamma \colon f_m \circ \cdots \circ f_1 \Longrightarrow h_n \circ \cdots \circ h_1$ is a $2$-cell in $A$ and $\delta \colon g_m\circ\cdots\circ g_1 \Longrightarrow k_n \circ \cdots \circ k_1$ is a $2$-cell in $B$ such that the following diagram commutes.
\begin{equation*}
\cd{
F(f_m \circ \cdots \circ f_1) \ar@{=>}[r]^-{F\gamma} \ar@{=>}[d]_-{\alpha_m \circ\cdots\circ \alpha_1} \ar@{}[dr]|{=} & F(h_n \circ \cdots \circ h_1) \ar@{=>}[d]^-{\beta_n\circ\cdots\circ \beta_1} \\
g_m\circ\cdots\circ g_1  \ar@{=>}[r]_-{\delta} & k_n \circ \cdots \circ k_1
}
\end{equation*}
The $2$-cell identities and composition operations in $\mathcal{S}(F)$ are induced from those in $A$ and $B$.

\begin{remark}
The strand $2$-category of a morphism of flexible $2$-categories $F \colon A \longrightarrow B$ is the universal flexible $2$-category equipped with a pair of morphisms $\mathcal{S}(F) \longrightarrow QA$ and $\mathcal{S}(F) \longrightarrow B$ and an invertible icon as in the following diagram.
\begin{equation*}
\cd{
\mathcal{S}(F) \ar[r] \ar[d] \ar@{}[dr]|{\cong} & B \ar[d]^-{\eta_B} \\
QA \ar[r]_-{QF} & QB
}
\end{equation*}
\end{remark}

Every morphism of flexible $2$-categories $F \colon A \longrightarrow B$ factors through its strand $2$-category via morphisms $J_F \colon A \longrightarrow \mathcal{S}(F)$ and $P_F \colon \mathcal{S}(F) \longrightarrow B$ which may be described as follows. The morphism $J_F\colon A \longrightarrow \mathcal{S}(F)$ is the identity on objects, sends an atomic morphism $f \colon a \longrightarrow b$ to the atomic morphism $(f, Ff, 1_{Ff})$, and sends a $2$-cell $\alpha$ to the pair $(\alpha,F\alpha)$. The morphism $P_F \colon \mathcal{S}(F) \longrightarrow B$ sends an object $a$ to $Fa$, sends an atomic morphism $(f,g,\alpha)$ to $g$, and sends a $2$-cell $(\gamma,\delta)$ to $\delta$. 

The strand $2$-category construction extends to a functor from $\mathbf{Fun}(\mathbb{D}^1,\mathbf{Flex})$ to $\mathbf{Fun}([2],\mathbf{Flex})$ which sends a commutative square in $\mathbf{Flex}$ as on the left below 
\begin{equation*}
\cd{
A \ar[d]_-F \ar[r]^-S \ar@{}[dr]|= & C \ar[d]^-G \\
B \ar[r]_-T & D
}
\qquad \qquad
\cd[@C=2.5em]{
A \ar@{}[dr]|= \ar[r]^-S \ar[d]_-{J_F} & C \ar[d]^-{J_G} \\
\mathcal{S}(F) \ar[r]^-{\mathcal{S}(S,T)} \ar[d]_-{P_F}  \ar@{}[dr]|= & \mathcal{S}(G) \ar[d]^-{P_G} \\
B \ar[r]_-T & D
}
\end{equation*}
to the commutative diagram  on the right above, where the morphism $\mathcal{S}(S,T) \colon \mathcal{S}(F) \longrightarrow \mathcal{S}(G)$ sends an object $a$ to $Sa$, a morphism $(f,g,\alpha)$ to $(Sf,Tg,T\alpha)$, and a $2$-cell $(\gamma,\delta)$ to $(S\gamma,T\delta)$.
This functor thus defines a functorial factorisation on $\mathbf{Flex}$, which is to say that it is a section of the composition functor. The following proposition shows that this functorial factorisation is an (inner anodyne, inner fibration) factorisation.

\begin{proposition} \label{innerprop}
Let $F \colon A \longrightarrow B$ be a morphism of flexible $2$-categories. Then $J_F \colon A \longrightarrow \mathcal{S}(F)$ is inner anodyne and $P_F \colon \mathcal{S}(F) \longrightarrow B$ is an inner fibration.
\end{proposition}
\begin{proof}
The morphism $J_F$ is evidently bijective on objects, faithful, and locally fully faithful. To see that it is essentially full, let $(f,g,\alpha) \colon a \longrightarrow b$ be an atomic or identity morphism in $\mathcal{S}(F)$, and let $f$ have atomic decomposition $f = f_n\circ\cdots\circ f_1$. Then $(1_f,\alpha) \colon J_Ff_n \circ \cdots \circ J_Ff_1 \cong (f,g,\alpha)$ is an invertible $2$-cell in $\mathcal{S}(F)$. Hence $J_F$ is inner anodyne. 

To see that $P_F$ is an inner fibration, let $(f_n,g_n,\alpha_n)\circ\cdots\circ(f_1,g_1,\alpha_1)$ be a composite of atomic morphisms in $\mathcal{S}(F)$ and let $\beta \colon g_n\circ\cdots\circ g_1 \cong h$ be an invertible $2$-cell in $B$ such that $h$ is atomic or an identity. Then $(1,\beta) \colon (f_n,g_n,\alpha_n)\circ\cdots\circ(f_1,g_1,\alpha_1) \cong (f_n\circ\cdots\circ f_1,h,\beta\cdot(\alpha_n\circ \cdots \circ \alpha_1))$ is an invertible $2$-cell in $\mathcal{S}(F)$ with atomic or identity codomain and is sent by $P_F$ to $\beta$. Hence $P_F$ is an inner fibration.
\end{proof}

We shall use the following proposition to show that the above functorial factorisation forms a weak factorisation system on $\mathbf{Flex}$.
\begin{proposition} \label{fillerprop}
A morphism of flexible $2$-categories $F \colon A \longrightarrow B$ is inner anodyne if and only if the lifting problem on the left below admits a solution, and is an inner fibration if and only if the lifting problem on the right below admits a solution.
\begin{equation*}
\cd{
A \ar[r]^-{J_F} \ar[d]_-F & \mathcal{S}(F) \ar[d]^-{P_F} \\
B \ar@{.>}[ur] \ar@{=}[r] & B
}
\qquad\qquad
\cd{
A \ar@{=}[r] \ar[d]_-{J_F} & A \ar[d]^-{F} \\
\mathcal{S}(F) \ar@{.>}[ur] \ar[r]_-{P_F} & B
}
\end{equation*}
\end{proposition}
\begin{proof}
Since the classes of inner anodyne extensions and inner fibrations are closed under retracts, it suffices by Proposition \ref{innerprop} to prove the ``only if'' directions. First, suppose that $F$ is inner anodyne and consider the lifting problem on the left above. 
We construct a solution $G \colon B \longrightarrow \mathcal{S}(F)$ as follows. For each object $b \in B$, we define $Gb$ to be the unique object of $A$ such that $FGb = b$. For each atomic morphism $g \colon b \longrightarrow c$ in $B$, if $g$ is in the image of $F$, let $f$ be the unique atomic morphism in $A$ such that $Ff = g$ and define $Gg = (f,g,1_g)$; otherwise, we choose a morphism $f \colon Gb \longrightarrow Gc$ in $A$ and an invertible $2$-cell $\alpha \colon Gf \cong g$ in $B$, which exist since $F$ is essentially full, and define $Gg = (f,g,\alpha)$. For each $2$-cell $\beta \colon g \Longrightarrow k$ in $B$, where $g$ and $k$ have atomic decompositions $g = g_m\circ \cdots \circ g_1$ and $k = k_n \circ \cdots \circ k_1$, let each $Gg_i = (f_i,g_i,\alpha_i)$ and $Gk_j = (h_j,k_j,\gamma_j)$; we define $G\beta$ to be the pair $(\delta,\beta)$ where $\delta \colon f_m \circ \cdots \circ f_1 \Longrightarrow h_n \circ \cdots \circ h_1$ is the unique $2$-cell in $A$ such that $F\delta$ is equal to the following composite.
\begin{equation*}
\cd{
F(f_m \circ \cdots \circ f_1) \ar@{=>}[rr]^-{\alpha_m \circ \cdots \circ \alpha_1} && g \ar@{=>}[r]^-{\beta} & k \ar@{=>}[rr]^-{\gamma_n^{-1} \circ \cdots \circ \gamma_1^{-1}} && F(h_n \circ \cdots \circ h_1)
}
\end{equation*}
This defines a solution to the lifting problem.

Now, suppose that $F$ is an inner fibration and consider the lifting problem on the right above. 
We construct a solution $H \colon \mathcal{S}(F) \longrightarrow A$ as follows. We define $H$ to be the identity on objects. For each atomic or identity morphism $(f,g,\alpha) \colon a \longrightarrow b$ of $\mathcal{S}(F)$,  if $\alpha$ is an identity, we define $H(f,g,\alpha) = f$; otherwise, we choose an atomic or identity morphism $h \colon a \longrightarrow b$ such that $Fh = g$ and an invertible $2$-cell $\beta \colon f \cong h$ such that $F\beta = \alpha$, as we may since $F$ is an inner fibration, and define $H(f,g,\alpha) = h$; we denote the $2$-cell $\beta$ by $\theta_{(f,g,\alpha)}$. Finally, for each $2$-cell $(\gamma,\delta) \colon (f_m,g_m,\alpha_m)\circ\cdots\circ(f_1,g_1,\alpha_1) \Longrightarrow (h_n,k_n,\beta_n)\circ\cdots\circ(h_1,k_1,\beta_1)$ in $\mathcal{S}(F)$, we define $H(\gamma,\delta)$ to be the following composite.
\begin{equation*}
\cd{
 f_m \circ \cdots \circ f_1 \ar@{=>}[r]^-{\gamma} & h_n\circ \cdots \circ h_1 \ar@{=>}[d]^-{\theta_{(h_n,k_n, \beta_n)}\circ \cdots\circ \theta_{(h_1,k_1,\beta_1)}} \\
H(f_m,g_m,\alpha_m) \circ \cdots \circ H(f_1,g_1,\alpha_1) \ar@{=>}[u]^-{\theta^{-1}_{(f_m,g_m,\alpha_m)} \circ \cdots \circ \theta^{-1}_{(f_1,g_1,\alpha_1)}} & H(h_n,k_n, \beta_n) \circ \cdots \circ H(h_1,k_1,\beta_1) \\
}
\end{equation*}
This defines a solution to the lifting problem.
\end{proof}

It is now a routine matter to complete the proof of the following theorem.

\begin{proposition} \label{innerwfsprop}
The inner anodyne extensions and inner fibrations form the left class and right class respectively of a weak factorisation system on $\mathbf{Flex}$.
\end{proposition}
\begin{proof}
The existence of factorisations is Proposition \ref{innerprop}, while the closure of the classes under retracts is evident. It remains to prove that the classes of inner anodyne extensions and inner fibrations are weakly orthogonal. Consider therefore a lifting problem
\begin{equation*}
\cd{
A \ar[d]_-F \ar[r]^-S & C \ar[d]^-G \\
B \ar@{..>}[ur] \ar[r]_-T & D
}
\end{equation*}
where $F$ is inner anodyne and $G$ is an inner fibration. We can produce a solution using the functorial factorisation provided by the strand $2$-category construction and the lifts provided by Proposition \ref{fillerprop}:\ take the composite morphism from $B$ to $C$ in the diagram below.
\begin{equation*}
\begin{gathered}[b]
\xymatrix@C=3em{
A \ar@{=}[r] \ar@{=}[d] & A \ar[r]^-S \ar[d]^-{J_F} & C \ar@{=}[r] \ar[d]_-{J_G} & C \ar[d]^-G \\
A \ar[d]_-F \ar[r]^-{J_F} & \mathcal{S}(F) \ar[r]^-{\mathcal{S}(S,T)} \ar[d]^-{P_F} & \mathcal{S}(G) \ar@{..>}[ur] \ar[r]_-{P_G} \ar[d]_-{P_G} & D \ar@{=}[d] \\
B \ar@{..>}[ur] \ar@{=}[r] & B \ar[r]_-T & D \ar@{=}[r] & D 
}\\[-\dp\strutbox]\end{gathered} 
\qedhere
\end{equation*}
\end{proof}

\subsection{Fibrations with fibrant codomain}
The fibrations in the left-induced model structure on $\mathbf{Flex}$ are, by definition, those morphisms of flexible $2$-categories that have the right lifting property in $\mathbf{Flex}$ with respect to all trivial cofibrations. While a more explicit description of arbitrary fibrations remains elusive, we are able to give several characterisations of the fibrations with fibrant codomain. (Compare the similar state of affairs for Joyal's model structure for quasi-categories.) We first give two preparatory lemmas and introduce some more terminology.

\begin{lemma} \label{innerfiblemma}
An inner fibration with fibrant codomain is a fibration in the left-induced model structure on $\mathbf{Flex}$ if and only if it has the right lifting property in $\mathbf{Flex}$ with respect to all trivial cofibrations between fibrant flexible $2$-categories.
\end{lemma}
\begin{proof}
Let $G \colon C \longrightarrow D$ be an inner fibration and suppose that $D$ is fibrant. If $G$ is a fibration, then by definition it has the right lifting property in $\mathbf{Flex}$ with respect to all trivial cofibrations. Conversely, suppose that $G$ has the right lifting property in $\mathbf{Flex}$ with respect to all trivial cofibrations between fibrant objects. Consider now a lifting problem in $\mathbf{Flex}$
\begin{equation*}
\cd{
A \ar[r]^-S \ar[d]_-F & C \ar[d]^-G \\
B \ar@{..>}[ur] \ar[r]_-T & D
}
\end{equation*}
where $F \colon A \longrightarrow B$ is a trivial cofibration. Since $D$ is fibrant, it has the right lifting property with respect to the trivial cofibration $\eta_B \colon B\longrightarrow QB$, and so there exists a morphism $U \colon QB \longrightarrow D$ making the diagram below commute.
\begin{equation*}
\cd{
A  \ar[rr]^-S \ar[d]_-F &  & C \ar[d]^-G \\
B \ar[r]^-{\eta_B} \ar@/_1pc/[rr]_-T & QB \ar@{..>}[r]^-U & D
}
\end{equation*}
Now, since $G$ is an inner fibration, it has the right lifting property with respect to the inner anodyne extension $\eta_A \colon A \longrightarrow QA$ by Proposition \ref{innerwfsprop},  so there exists a morphism $V \colon QA \longrightarrow C$ making the diagram below commute.
\begin{equation*}
\cd{
A \ar[r]_-{\eta_A} \ar@/^1pc/[rr]^-S \ar[d]_-F & QA \ar[d]_-{QF} \ar@{..>}[r]_-V & C \ar[d]^-G \\
B \ar[r]^-{\eta_B} \ar@/_1pc/[rr]_-T & QB \ar[r]^-U & D
}
\end{equation*}
But $QF \colon QA \longrightarrow QB$ is a trivial cofibration between fibrant objects by inspection, with respect to which therefore $G$ has the right lifting property by assumption, so there exists a morphism $W \colon QB \longrightarrow C$ making the diagram below commute.
\begin{equation*}
\cd{
A \ar[r]_-{\eta_A} \ar@/^1pc/[rr]^-S \ar[d]_-F & QA \ar[d]_-{QF} \ar[r]_-V & C \ar[d]^-G \\
B \ar[r]^-{\eta_B} \ar@/_1pc/[rr]_-T & QB \ar@{..>}[ur]|-W \ar[r]^-U & D
}
\end{equation*}
Hence the composite $W \circ \eta_B \colon B \longrightarrow C$ is a solution to the original lifting problem. Therefore $G$ is a fibration.
\end{proof}

We say that a normal pseudofunctor between bicategories is an \emph{injective biequivalence} if it is injective on objects, faithful, essentially full, and locally fully faithful, while we say that it is an \emph{equifibration} if it has the equivalence lifting property and is an isofibration on hom-categories. One can show that a normal pseudofunctor is an equifibration if and only if it has the right lifting property in $\mathbf{Bicat}$ with respect to the source inclusions $\mathbb{D}^0 \longrightarrow \mathbb{I}$ and $\mathbb{D}^1 \longrightarrow \mathbb{I}^2$, where $\mathbb{I}$ denotes the free-living isomorphism.

\begin{lemma} \label{bicatlemma}
The injective biequivalences and equifibrations form the left class and right class respectively of a weak factorisation system on $\mathbf{Bicat}$. 
\end{lemma}
\begin{proof}
One can prove this lemma by following the strategy of the proof of Proposition \ref{innerwfsprop}, replacing the strand $2$-category construction by the following construction. For each normal pseudofunctor between bicategories $F \colon A \longrightarrow B$, we define  $\mathcal{P}(F)$ to be the bicategory whose objects are triples $(a,b,f)$, where $a$ is an object of $A$, $b$ is an object of $B$, and $f \colon Fa \longrightarrow b$ is an equivalence in $B$; a morphism $(a,b,f) \longrightarrow (c,d,g)$ is a triple $(s,t,\alpha)$, where $s \colon a \longrightarrow c$ is a morphism in $A$, $t \colon b \longrightarrow d$ is a morphism in $B$, and $\alpha \colon g\circ Fs \cong t \circ f$ is an invertible $2$-cell in $B$; a $2$-cell $(s,t,\alpha) \Longrightarrow (u,v,\beta) \colon (a,b,f) \longrightarrow (c,d,g)$ is a pair $(\gamma,\delta)$ where $\gamma \colon s \Longrightarrow u$ is a $2$-cell in $A$ and $\delta \colon t \Longrightarrow v$ is a $2$-cell in $B$ such that the following diagram commutes.
\begin{equation*}
\cd[@C=2.5em]{
g \circ Fs \ar@{=>}[d]_-{\alpha} \ar@{=>}[r]^-{1 \circ F\gamma} \ar@{}[dr]|{=} & g \circ Fu  \ar@{=>}[d]^-{\beta} \\
t \circ f \ar@{=>}[r]_-{\delta \circ 1} & v \circ f
}
\end{equation*}
Composition operations, identities, and coherence constraints are induced from those of the bicategories $A$ and $B$ and the normal pseudofunctor $F$.

One can proceed to show that this construction extends to a functorial factorisation $F \mapsto P_F \circ J_F$ on $\mathbf{Bicat}$, that $J_F$ is an injective biequivalence and $P_F$ is an equifibration, and that a normal pseudofunctor $F$ is an injective biequivalence (resp.\ equifibration) if and only if the lifting problem on the left (resp.\ right) below admits a solution. 
\begin{equation*}
\cd{
A \ar[r]^-{J_F} \ar[d]_-F & \mathcal{P}(F) \ar[d]^-{P_F} \\
B \ar@{.>}[ur] \ar@{=}[r] & B
}
\qquad\qquad
\cd{
A \ar@{=}[r] \ar[d]_-{J_F} & A \ar[d]^-{F} \\
\mathcal{P}(F) \ar@{.>}[ur] \ar[r]_-{P_F} & B
}
\end{equation*}
Weak orthogonality of the two classes can then be shown as in the proof of Proposition \ref{innerwfsprop}.
\end{proof}

We say that a morphism in a flexible $2$-category is a \emph{basic equivalence} if it is atomic or an identity morphism and it has a pseudo-inverse that is also an atomic or identity morphism. Every equivalence in a fibrant flexible $2$-category is isomorphic to a basic equivalence. We say that a morphism of flexible $2$-categories $F \colon A \longrightarrow B$ has the \emph{basic equivalence lifting property} if for every object $a$ in $A$ and every basic equivalence $g \colon Fa \longrightarrow b$ in $B$, there exists a basic equivalence $f \colon a \longrightarrow a'$ in $A$ such that $Ff = g$. 

As in \cite[Proposition 6]{Lac04}, one can show that a morphism of flexible $2$-categories has the basic equivalence lifting property if and only if it has the right lifting property in $\mathbf{Flex}$ with respect to the source inclusion $\mathbb{D}^0 \longrightarrow \mathbb{E}$, where $\mathbb{E}$ denotes the free $2$-category containing an adjoint equivalence; note that $\mathbb{E} \cong Q\mathbb{I}$.

We are now ready to state our characterisations of the fibrations with fibrant codomain in the left-induced model structure on $\mathbf{Flex}$. 

\begin{theorem}
Let $F \colon A \longrightarrow B$ be a morphism of flexible $2$-categories and suppose that $B$ is fibrant in the left-induced model structure on $\mathbf{Flex}$. The following are equivalent:
\begin{enumerate}[\normalfont(i)]
\item $F$ is a fibration in the left-induced model structure on $\mathbf{Flex}$;
\item $F$ is a retract in $\mathbf{Fun}(\mathbb{D}^1,\mathbf{Flex})$ of $QG$ for some equifibration $G$ in $\mathbf{2Cat}$;
\item $F$ has the right lifting property in $\mathbf{Flex}$ with respect to the source inclusion $\mathbb{D}^0 \longrightarrow \mathbb{E}$ and the unit morphism $\eta_{[2]} \colon [2] \longrightarrow Q[2]$;
\item $F$ has the basic equivalence lifting property and is an inner fibration;
\item $A$ is fibrant and $F$ has the right lifting property in $\mathbf{Flex}$ with respect to the source inclusions $\mathbb{D}^0 \longrightarrow \mathbb{E}$ and $\mathbb{D}^1 \longrightarrow \mathbb{I}^2$;
\item $F$ is isomorphic in  $\mathbf{Fun}(\mathbb{D}^1,\mathbf{Flex})$ to $QG$ for some equifibration $G$ in $\mathbf{Bicat}$.
\end{enumerate}
\end{theorem}
\begin{proof}
(i)$\implies$(ii). First, factorise $F$ in $\mathbf{2Cat}$ as $F = H\circ G$, where $G \colon A \longrightarrow C$ is a trivial cofibration and $H \colon C \longrightarrow B$ is an equifibration. Since the functor $Q \colon \mathbf{2Cat} \longrightarrow \mathbf{Flex}$ preserves cofibrations by direct inspection and preserves biequivalences by Theorem \ref{quillenthm}, it follows that $QG$ is a trivial cofibration. Since $B$ is a fibrant flexible $2$-category, there exists a retraction $R \colon QB \longrightarrow B$ of $\eta_B$ by Theorem \ref{fibprop}. Then, since $F$ is a fibration in the left-induced model structure, it has the right lifting property in $\mathbf{Flex}$ with respect to the trivial cofibration $QG \circ \eta_A$, whence there exists a morphism $S \colon QC \longrightarrow A$ making the following diagram commute.
\begin{equation*}
\cd[@C=2.5em]{
A \ar@/^1pc/[rr]^-{1_A} \ar[r]_-{QG \circ \eta_A} \ar[d]_-F & QC \ar[r]_-S \ar[d]^-{QH} & A \ar[d]^-F \\
B \ar[r]^-{\eta_B} \ar@/_1pc/[rr]_-{1_B} \ar[r]^-{\eta_B} & QB \ar[r]^-R & B
}
\end{equation*}
Hence $F$ is a retract of $QH$ in $\mathbf{Fun}(\mathbb{D}^1,\mathbf{Flex})$.

(ii)$\implies$(iii). By Theorem \ref{quillenthm}, $QG$ is a fibration in the left-induced model structure. Hence its retract $F$ is also a fibration, since the fibrations in a model category are closed under retracts. Therefore $F$ has the right lifting property in $\mathbf{Flex}$ with respect to the trivial cofibrations $\mathbb{D}^0 \longrightarrow \mathbb{E}$ and $\eta_{[2]} \colon [2] \longrightarrow Q[2]$.

(iii)$\implies$(iv). The morphism $F$ has the basic equivalence lifting property since it has the right lifting property in $\mathbf{Flex}$ with respect to the source inclusion $\mathbb{D}^0 \longrightarrow \mathbb{E}$, and is an inner fibration by Proposition \ref{innerfibprop}, since it has the right lifting property with respect to $\eta_{[2]} \colon [2] \longrightarrow Q[2]$ and $B$ is fibrant. 

(iv)$\implies$(v). Since $F$ is an inner fibration and $B$ is fibrant, it follows from Proposition \ref{innerfibprop} that $A$ is fibrant and that $F$ has the right lifting property with respect to the source inclusion $\mathbb{D}^1 \longrightarrow \mathbb{I}^2$. Moreover, $F$ has the right lifting property with respect to $\mathbb{D}^0 \longrightarrow \mathbb{E}$ since it has the basic equivalence lifting property.

(v)$\implies$(vi). Since $A$ and $B$ are fibrant, $F$ is isomorphic in $\mathbf{Fun}(\mathbb{D}^1, \mathbf{Flex})$ to $QG$ for some normal pseudofunctor between bicategories $G$ by Theorem \ref{bicatthm}. But the source inclusions $\mathbb{D}^0 \longrightarrow \mathbb{E}$ and $\mathbb{D}^1 \longrightarrow \mathbb{I}^2$ are isomorphic to the images under $Q$ of the source inclusions $\mathbb{D}^0 \longrightarrow \mathbb{I}$ and $\mathbb{D}^1 \longrightarrow \mathbb{I}^2$. Hence the normal pseudofunctor $G$ has the right lifting property in $\mathbf{Bicat}$ with respect to the source inclusions $\mathbb{D}^0 \longrightarrow \mathbb{I}$ and $\mathbb{D}^1 \longrightarrow \mathbb{I}^2$. Therefore  $G$ is an equifibration.

(vi)$\implies$(i). It suffices to show that, for each equifibration $G$ in $\mathbf{Bicat}$, the morphism $QG$ is a fibration in the left-induced model structure on $\mathbf{Flex}$. Since $G$ is an isofibration on hom-categories, it has the right lifting property in $\mathbf{Bicat}$ with respect to the source inclusion $\mathbb{D}^1 \longrightarrow \mathbb{I}^2$, from which it follows that $QG$ is an inner fibration by Proposition \ref{innerfibprop}. Note also that the functor $Q \colon \mathbf{Bicat} \longrightarrow \mathbf{Flex}$ sends a normal pseudofunctor to a trivial cofibration if and only if it is an injective biequivalence. Hence, by Lemma \ref{innerfiblemma} and Theorem \ref{bicatthm}, $QG$ is a fibration in the left-induced model structure on $\mathbf{Flex}$ if and only if $G$ has the right lifting property in $\mathbf{Bicat}$ with respect to all injective biequivalences. The result now follows from Lemma \ref{bicatlemma}.
\end{proof}

\begin{remark} \label{properremark}
Since every flexible $2$-category is cofibrant, the left-induced model structure on $\mathbf{Flex}$ is left proper by \cite[Corollary 13.1.3]{hirschhorn}. However, it is not right proper. For the unit morphism $\eta_{[2]} \colon [2] \longrightarrow Q[2]$ displayed in \S\ref{qsec} is a biequivalence, but its pullback along the fibration $\mathbb{D}^1 \longrightarrow Q[2]$ that picks out the edge $02$ is the boundary inclusion $\partial\mathbb{D}^1 \longrightarrow \mathbb{D}^1$, which is not a biequivalence. 
\end{remark}

\section{Monoidal structures} \label{monenr}
In this final section, we construct two symmetric monoidal closed structures on $\mathbf{Flex}$, the Gray symmetric monoidal closed structure and the cartesian closed structure, and prove that both are compatible with the left-induced model structure.

\subsection{The Gray symmetric monoidal closed structure}
In addition to its cartesian closed structure, the category $\mathbf{2Cat}$ also admits a symmetric monoidal closed structure introduced by Gray \cite[\S4.24]{gray}. The unit object is the terminal $2$-category $\mathbb{D}^0$, and the tensor product of a pair of $2$-categories $A$ and $B$ can be constructed as the domain of the $U$-cartesian lift
\begin{equation*}
(\pi_1,\pi_2) \colon A \otimes B \longrightarrow A \times B
\end{equation*}
of the canonical functor
\begin{equation*}
(\pi_1,\pi_2) \colon UA \mathbin{\square} UB \longrightarrow UA \times UB,
\end{equation*}
where $U$ denotes the forgetful functor $\mathbf{2Cat} \longrightarrow \mathbf{Cat}$. For example, the symmetric Gray tensor product $\mathbb{D}^1 \otimes \mathbb{D}^1$ is the free $2$-category presented by the following diagram.
\begin{equation*}
\cd{
\bullet \ar[r] \ar[d] \ar@{}[dr]|{\cong} & \bullet \ar[d] \\
\bullet \ar[r] & \bullet
}
\end{equation*}

\begin{remark}
See \cite{bourkegurski} for an elegant proof that this definition of the symmetric Gray tensor product is equivalent to the usual definition by generators and relations.
\end{remark}

The internal hom for a pair of $2$-categories $A$ and $B$ in this symmetric monoidal closed structure on $\mathbf{2Cat}$ is the $2$-category $[A,B]_{\mathbf{2Cat}}$ of $2$-functors $A \longrightarrow B$, pseudonatural transformations between them, and modifications between those. Note that there is a a canonical forgetful functor $U[A,B]_{\mathbf{2Cat}} \longrightarrow [UA,UB]_{\mathbf{Cat}}$ for each pair of $2$-categories $A$ and $B$, whose codomain is as in \S\ref{funnysec}.

Since the funny symmetric monoidal structure on $\mathbf{Cat}$ restricts to a symmetric monoidal structure on $\mathbf{Free}$, it follows that the Gray symmetric monoidal structure on $\mathbf{2Cat}$ restricts to a symmetric monoidal structure on $\mathbf{Flex}$. It likewise follows formally that this symmetric monoidal structure on $\mathbf{Flex}$ is closed. The internal hom $[A,B]_{\mathbf{Flex}}$ for a pair of flexible $2$-categories $A$ and $B$ in this symmetric monoidal closed structure on $\mathbf{Flex}$ can be constructed as the domain of the $U$-cartesian lift
\begin{equation*}
\cd{
[A,B]_{\mathbf{Flex}} \ar[r] &  [A,B]_{\mathbf{2Cat}}
}
\end{equation*}
of the top morphism in the following pullback square in $\mathbf{Cat}$.
\begin{equation*}
\cd{
U[A,B]_{\mathbf{Flex}} \ar[r] \ar[d]_-U \fatterpullbackcorner & U[A,B]_{\mathbf{2Cat}} \ar[d]^-U \\
[UA,UB]_{\mathbf{Free}} \ar[r] & [UA,UB]_{\mathbf{Cat}}
}
\end{equation*}
Thus the flexible $2$-category $[A,B]_{\mathbf{Flex}}$ has as objects the morphisms of flexible $2$-categories $A \longrightarrow B$, its atomic and identity morphisms are the pseudonatural transformations whose $1$-cell components are atomic or identity morphisms, and its $2$-cells are modifications between composite pseudonatural transformations.

Gray's symmetric monoidal closed structure on $\mathbf{2Cat}$ is compatible with Lack's model structure in the sense that together they  form a symmetric monoidal model category \cite[Theorem 7.5]{Lac02}. 
The following two theorems are now immediate consequences of the definition of the left-induced model structure on $\mathbf{Flex}$. We refer to \cite[\S4.2]{hovey} for the notions of symmetric monoidal model category and symmetric monoidal Quillen equivalence.

\begin{theorem} \label{monmodthm1}
The category $\mathbf{Flex}$, equipped with the Gray symmetric monoidal closed structure and the left-induced model structure, is a symmetric monoidal model category.
\end{theorem}
\begin{proof}
This follows from \cite[Theorem 7.5]{Lac02} since the inclusion functor $\mathbf{Flex} \longrightarrow \mathbf{2Cat}$ is strict monoidal and creates cofibrations, trivial cofibrations, and colimits.
\end{proof}

\begin{theorem} \label{monadjthm1}
The adjunction 
\begin{equation*}
\xymatrix{
\mathbf{Flex} \ar@<1.5ex>[rr]_-{\hdash}&& \ar@<1.5ex>[ll]^-{Q} \mathbf{2Cat}
}
\end{equation*}
is a symmetric monoidal Quillen equivalence with respect to the Gray symmetric monoidal structures, the left-induced model structure on $\mathbf{Flex}$, and Lack's model structure on $\mathbf{2Cat}$.
\end{theorem}
\begin{proof}
This follows from the Quillen equivalence of Theorem \ref{quillenthm} and the fact that the inclusion functor $\mathbf{Flex} \longrightarrow \mathbf{2Cat}$ is strict symmetric monoidal.
\end{proof}

It follows from Theorem \ref{monmodthm1} that, for every flexible $2$-category $A$ and fibrant flexible $2$-category $B$, the internal hom $[A,B]_{\mathbf{Flex}}$ is a fibrant flexible $2$-category. Likewise, it follows from Theorem \ref{monadjthm1} that, for every flexible $2$-category $A$ and $2$-category $B$, there is a canonical isomorphism $Q[A,B]_{\mathbf{2Cat}} \cong [A,QB]_{\mathbf{Flex}}$.  Moreover, one can show that, for each pair of bicategories $A$ and $B$, the internal hom $[QA,QB]_{\mathbf{Flex}}$ is isomorphic to the path $2$-category of the bicategory $\mathbf{Hom}(A,B)$, whose objects are normal pseudofunctors $A \longrightarrow B$, whose morphisms are pseudonatural transformations, and whose $2$-cells are modifications.

\subsection{The cartesian closed structure}
While the existence of the Gray symmetric monoidal closed structure on $\mathbf{Flex}$ could be proved more or less formally, there seems to be no alternative but to construct the cartesian internal homs of $\mathbf{Flex}$, and prove their universal property, by hand. Our proof hinges on the relationship between the symmetric Gray tensor product and the cartesian product of flexible $2$-categories, as given in the second of the following two lemmas.

\begin{lemma} \label{compbiequiv}
For each pair of flexible $2$-categories $A$ and $B$, the canonical morphism $(\pi_1,\pi_2) \colon A \otimes B \longrightarrow A \boxtimes B$ is a biequivalence. 
\end{lemma}
\begin{proof}
It is immediate from their constructions that the canonical morphisms $A \otimes B \longrightarrow A \times B$ and $A \boxtimes B \longrightarrow A \times B$ are biequivalences. But the former is equal to the composite of the latter with the morphism in the statement of the lemma, which is therefore a biequivalence.
\end{proof}

\begin{lemma} \label{cubicallemma}
For each pair of flexible $2$-categories $A$ and $B$, the diagram below is a pushout square in $\mathbf{Flex}$,
\begin{equation*}
\cd[@C=3.5em]{
\displaystyle\sum_{\substack{f,g \\ \mathrm{atomic}}} \mathbb{D}^1 \otimes \mathbb{D}^1 \ar[r]^-{\sum (\pi_1,\pi_2)} \save[]-<0pt,+10pt>{} \ar[d]^-{(f\otimes g)} \restore & \displaystyle\sum_{\substack{f,g \\ \mathrm{atomic}}} \mathbb{D}^1 \boxtimes \mathbb{D}^1 \save[]-<0pt,+10pt>{} \ar[d]^-{(f\boxtimes g)} \restore \\
A \otimes B \ar[r]_-{(\pi_1,\pi_2)} & A \boxtimes B \fatpushoutcorner
}
\end{equation*}
in which the coproducts are indexed by pairs of atomic morphisms $f$ in $A$ and $g$ in $B$.
\end{lemma}
\begin{proof}
Let $C \longrightarrow A \boxtimes B$ denote the pushout-corner morphism of this square. The final observation of \S\ref{funnysec} implies that this morphism is an isomorphism on underlying categories. It is moreover locally fully faithful, since it composes with the pushout injection $A \otimes B \longrightarrow C$, which is a pushout of a trivial cofibration and hence a biequivalence, to give the biequivalence $A \otimes B \longrightarrow A \boxtimes B$ of Lemma \ref{compbiequiv}. Hence the pushout-corner morphism is an isomorphism and the square is a pushout.
\end{proof}

A morphism of flexible $2$-categories $A \boxtimes B \longrightarrow C$ therefore amounts to a morphism of flexible $2$-categories $F \colon A \otimes B \longrightarrow C$ together with the following additional data:\ for each pair of atomic morphisms $f \colon a \longrightarrow a'$ in $A$ and $g \colon b \longrightarrow b'$ in $B$, an atomic or identity morphism $F(f,g) \colon F(a,b) \longrightarrow F(a',b')$ in $C$ and an invertible $2$-cell
\begin{equation*}
\cd{
F(a,b) \ar[rr]^-{F(f,g)} \ar[dr]_-{F(a,g)} &\dtwocell[0.4]{d}{F^+_{f,g}} & F(a',b') \\
& F(a,b') \ar[ur]_-{F(f,b')}
}
\end{equation*}
in $C$. We may therefore leverage the natural bijection between $2$-functors $A \otimes B \longrightarrow C$ and cubical functors $A \times B \longrightarrow C$ (see \cite[\S4]{GPS}) to deduce that morphisms of flexible $2$-categories $A \boxtimes B \longrightarrow C$ are equivalent to the following notion.
\begin{definition}
Let $A$, $B$, and $C$ be flexible $2$-categories. A \emph{two-variable morphism of flexible $2$-categories} $F \colon (A,B) \longrightarrow C$ is a normal pseudofunctor $F \colon A \times B \longrightarrow C$ such that: for each object $a\in A$, the normal pseudofunctor $F(a,-) \colon B \longrightarrow C$ is a $2$-functor,
  for each object $b \in B$, the normal pseudofunctor $F(-,b) \colon A \longrightarrow C$ is a $2$-functor, and
  for each pair of atomic or identity morphisms $f$ in $A$ and $g$ in $B$, the morphism $F(f,g)$ is an atomic or identity morphism in $C$.
\end{definition}
We write $\mathbf{Flex}(A,B;C)$ for the set of two-variable morphisms of flexible $2$-categories $(A,B) \longrightarrow C$. This defines a functor $\mathbf{Flex}^{\mathrm{op}} \times \mathbf{Flex}^{\mathrm{op}} \times \mathbf{Flex} \longrightarrow \mathbf{Set}$ by composition and substitution. We may now state the following proposition, whose proof is a straightforward verification. 

\begin{proposition}
There is a bijection
$$\mathbf{Flex}(A\boxtimes B,C) \cong \mathbf{Flex}(A,B;C)$$
natural in $A,B,C \in \mathbf{Flex}$.
\end{proposition}

In order to construct the cartesian internal homs in $\mathbf{Flex}$, we first recall the \textit{recherché} notion of enhanced pseudonatural transformation (see \cite[pp.\ 6--7]{Lac06}), which arises in the description of the internal homs in the cartesian closed category $\mathbf{Bicat}$.

\begin{definition}
Let $F,G \colon A \longrightarrow B$ be normal pseudofunctors between bicategories. An \emph{enhanced pseudonatural transformation} from $F$ to $G$ is a normal pseudofunctor $H \colon \mathbb{D}^1 \times A \longrightarrow B$ whose restrictions along the inclusions $A \cong \{0\} \times A \longrightarrow \mathbb{D}^1 \times A$ and $A \cong \{1\}\times A \longrightarrow \mathbb{D}^1 \times A$ are $F$ and $G$ respectively.
\end{definition}

An enhanced pseudonatural transformation between normal pseudofunctors $F,G \colon A \longrightarrow B$ thus amounts to a pseudonatural transformation $\eta \colon F \Longrightarrow G$ together with the following ``enhancement'':\ for each morphism $f \colon x \longrightarrow y$ in $A$, a morphism $\eta_f^0 \colon Fx \longrightarrow Gy$ and an invertible $2$-cell
\begin{equation*}
\cd{
Fx \ar[rr]^-{\eta_f^0} \ar[dr]_-{Ff}&\dtwocell[0.4]{d}{\eta_f^+}& Gy \\
& Fy \ar[ur]_-{\eta_y} 
}
\end{equation*}
in $B$ such that, for each object $x \in A$, $\eta_{1_x}^0 = \eta_x$ and $ \eta_{1_x}^+ = 1_{\eta_x^0}$

The cartesian internal homs in $\mathbf{Bicat}$ may be described as follows. For each pair of bicategories $A,B$, the cartesian internal hom $\mathbf{Fun}(A,B)$ is the bicategory whose objects are the normal pseudofunctors $A \longrightarrow B$, whose morphisms are the enhanced pseudonatural transformations between them, and whose $2$-cells are the modifications between their underlying pseudonatural transformations. One may either directly define the vertical composition of enhanced pseudonatural transformations and prove it coherently associative, or one may use the fact that every pseudonatural transformation admits an enhancement to define the bicategory structure of $\mathbf{Fun}(A,B)$ by means of transport of structure from the bicategory $\mathbf{Hom}(A,B)$. Either way, forgetting enhancements defines a bijective-on-objects biequivalence $\mathbf{Fun}(A,B) \longrightarrow \mathbf{Hom}(A,B)$.

We now turn to the cartesian closed structure of $\mathbf{Flex}$. For each pair of flexible $2$-categories $A$ and $B$, we define a flexible $2$-category $\underline{\mathbf{Flex}}(A,B)$ as follows. The objects of its underlying free category $U\underline{\mathbf{Flex}}(A,B)$ are the morphisms of flexible $2$-categories $A \longrightarrow B$, and its atomic and identity morphisms $\eta\colon F \Longrightarrow G$ are the enhanced pseudonatural transformations such that for every object $x \in A$ and every atomic morphism $f$ in $A$, the morphisms $\eta_x$ and $\eta_f$ are atomic or identity morphisms in $B$.  There is an evident functor $U\underline{\mathbf{Flex}}(A,B) \longrightarrow U[A,B]_{\mathbf{2Cat}}$ which is an inclusion on objects and which forgets enhancements on atomic morphisms. We  define the flexible $2$-category $\underline{\mathbf{Flex}}(A,B)$ to be the domain of the $U$-cartesian lift
\begin{equation*}
\underline{\mathbf{Flex}}(A,B) \longrightarrow [A,B]_{\mathbf{2Cat}}
\end{equation*}
of the functor
\begin{equation*}
U\underline{\mathbf{Flex}}(A,B) \longrightarrow U[A,B]_{\mathbf{2Cat}}.
\end{equation*}
Thus a $2$-cell in $\underline{\mathbf{Flex}}(A,B)$ between a parallel pair of paths of atomic morphisms is a modification between the underlying composite pseudonatural transformations.

It is now a matter of unravelling the definitions to prove that morphisms of flexible $2$-categories $A \longrightarrow \underline{\mathbf{Flex}}(B,C)$ are in natural bijection with two-variable morphisms of flexible $2$-categories $(A,B) \longrightarrow C$. We may therefore deduce the following proposition. 

\begin{proposition}
The category $\mathbf{Flex}$ is cartesian closed.
\end{proposition}
\begin{proof}
This follows from the natural bijections
\begin{equation*}
\mathbf{Flex}(A\boxtimes B, C) \cong \mathbf{Flex}(A,B;C) \cong \mathbf{Flex}(A,\underline{\mathbf{Flex}}(B,C)).  \qedhere
\end{equation*}
\end{proof}

Having established the cartesian closed structure of $\mathbf{Flex}$, we may now prove that it is compatible with the left-induced model structure. Note that a model category is said to be cartesian if it is a cartesian closed category and a monoidal model category with respect to its cartesian closed structure.

\begin{theorem} \label{cartesianthm}
The category $\mathbf{Flex}$, equipped with the left-induced model structure, is a cartesian model category.
\end{theorem}
\begin{proof}
Since every object is cofibrant in the left-induced model structure on $\mathbf{Flex}$, it remains to show that the pushout-product in $\mathbf{Flex}$ of two cofibrations is a cofibration, and that cartesian product with any object preserves biequivalences. The former follows from the fact that the pushout-product in $\mathbf{Free}$ of two injective morphisms is an injective morphism, while the latter follows from the facts that, for any pair of flexible $2$-categories $A$ and $B$,  the canonical $2$-functor $A \boxtimes B \longrightarrow A \times B$ is a biequivalence by construction, and that for any $2$-category $C$, the functor $C \times (-) \colon \mathbf{2Cat} \longrightarrow \mathbf{2Cat}$ preserves biequivalences by \cite[Lemma 7.4]{Lac02}.
\end{proof}

It follows that,  for every flexible $2$-category $A$ and fibrant flexible $2$-category $B$, the flexible $2$-category $\underline{\mathbf{Flex}}(A,B)$ is fibrant. In particular, 
Theorem \ref{bicatthm} implies that, for each pair of bicategories $A$ and $B$, the flexible $2$-category $\underline{\mathbf{Flex}}(QA,QB)$ is isomorphic to the path $2$-category of the cartesian internal hom $\mathbf{Fun}(A,B)$ in $\mathbf{Bicat}$.

Our final theorem relates the two monoidal model structures on  $\mathbf{Flex}$. Since the unit object of the Gray symmetric monoidal structure is terminal, it follows that the canonical morphisms $(\pi_1,\pi_2) \colon A \otimes B \longrightarrow A \boxtimes B$ upgrade the identity functor on $\mathbf{Flex}$ to both a symmetric monoidal functor $1 \colon (\mathbf{Flex},\boxtimes) \longrightarrow (\mathbf{Flex},\otimes)$ and a  symmetric opmonoidal functor $1 \colon (\mathbf{Flex},\otimes) \longrightarrow (\mathbf{Flex},\boxtimes)$ between the indicated symmetric monoidal categories. We refer to \cite[Definition 3.6]{schwede} for the notion of a weak (symmetric) monoidal Quillen equivalence.

\begin{theorem}
The identity adjunction
\begin{equation*}
\xymatrix{
(\mathbf{Flex},\otimes) \ar@<1.5ex>[rr]_-{\hdash}^-1&& \ar@<1.5ex>[ll]^-{1} (\mathbf{Flex},\boxtimes)
}
\end{equation*}
is a weak symmetric monoidal Quillen equivalence with respect to the left-induced model structure and the indicated symmetric monoidal closed structures. \end{theorem}
\begin{proof}
This is a combination of Theorem \ref{quillenthm} and Lemma \ref{compbiequiv}.
\end{proof}
It follows that, for every flexible $2$-category $A$ and every fibrant flexible $2$-category $B$, the canonical morphism $\underline{\mathbf{Flex}}(A,B) \longrightarrow [A,B]_{\mathbf{Flex}}$ is a bijective-on-objects biequivalence. Hence, for every flexible $2$-category $A$ and every $2$-category $B$, there is a canonical bijective-on-objects biequivalence $\underline{\mathbf{Flex}}(A,QB) \longrightarrow Q[A,B]_{\mathbf{2Cat}}$, and hence also a biequivalence $\underline{\mathbf{Flex}}(A,QB) \longrightarrow \mathbf{Hom}(A,B)$.


\begin{thebibliography}{GKR20}

\bibitem[AR94]{arbook}
Jiří Adámek and Jiří Rosický. \textit{Locally presentable and accessible categories}. London Mathematical Society Lecture Note Series, 189. Cambridge University Press, Cambridge, 1994. {\rm xiv}+316 pp.

\bibitem[Bou21]{bourke}
John Bourke. Accessible aspects of 2-category theory. \textit{J. Pure Appl. Algebra} \textbf{225} (2021), no. 3, Paper No. 106519, 43 pp.

\bibitem[BG17]{bourkegurski}
John Bourke and Nick Gurski. The Gray tensor product via factorisation. \textit{Appl. Categ. Structures} \textbf{25} (2017), no. 4, 603--624.

\bibitem[BH22]{bourkehenry}
John Bourke and Simon Henry. Algebraically cofibrant and fibrant objects revisited. \textit{Homology Homotopy Appl.} \textbf{24} (2022), no. 1, 271--298.

\bibitem[CR14]{chingriehl}
Michael Ching and Emily Riehl. Coalgebraic models for combinatorial model categories. \textit{Homology Homotopy Appl.} \textbf{16} (2014), no. 2, 171--184.

\bibitem[Gar12]{smallobject}
Richard Garner. Understanding the small object argument. \textit{Appl. Categ. Structures} \textbf{20} (2012), no. 2, 103--141.

\bibitem[GKR20]{gmr}
Richard Garner, Magdalena Kędziorek, and Emily Riehl. Lifting accessible model structures. \textit{J. Topol.} \textbf{13} (2020), no. 1, 59--76.

\bibitem[GPS95]{GPS}
R.\ Gordon, A.\ J.\ Power, and Ross Street. Coherence for tricategories. \textit{Mem. Amer. Math. Soc.} \textbf{117} (1995), no. 558, {\rm vi}+81 pp.

\bibitem[Gra74]{gray}
John W.\ Gray. \textit{Formal category theory: adjointness for $2$-categories}. Lecture Notes in Mathematics, Vol. 391. Springer-Verlag, Berlin-New York, 1974. {\rm xii}+282 pp.

\bibitem[Gur13]{Gur13}
Nick Gurski. \textit{Coherence in three-dimensional category theory}. Cambridge Tracts in Mathematics, 201. Cambridge University Press, Cambridge, 2013. viii+278 pp.

\bibitem[Hir03]{hirschhorn}
Philip S.\ Hirschhorn. \textit{Model categories and their localizations}. Mathematical Surveys and Monographs, 99. American Mathematical Society, Providence, RI, 2003. xvi+457 pp.

\bibitem[Hov99]{hovey}
Mark Hovey. \textit{Model categories}. Mathematical Surveys and Monographs, 63. American Mathematical Society, Providence, RI, 1999. xii+209 pp.

\bibitem[JS93]{pspb}
André Joyal and Ross Street. Pullbacks equivalent to pseudopullbacks. \textit{Cahiers Topologie Géom. Différentielle Catég.} \textbf{34} (1993), no. 2, 153--156.

\bibitem[Lac02]{Lac02}
Stephen Lack. A Quillen model structure for 2-categories. \textit{$K$-Theory} \textbf{26} (2002), no. 2, 171--205.

\bibitem[Lac04]{Lac04}
Stephen Lack. A Quillen model structure for bicategories. \textit{$K$-Theory} \textbf{33} (2004), no. 3, 185--197.

\bibitem[Lac06]{Lac06}
Stephen Lack. A convenient 2-category of bicategories. Talk at CT2006, the 2006 International Category Theory Conference. Slides available at \url{https://www.mscs.dal.ca/~selinger/ct2006/slides/CT06-Lack.pdf}.

\bibitem[LP08]{lackpaoli}
Stephen Lack and Simona Paoli. 2-nerves for bicategories. \textit{$K$-Theory} \textbf{38} (2008), no. 2, 153--175.

\bibitem[LR12]{LR12}
Stephen Lack and Ji\v{r}\'{\i} Rosick\'{y}. Enriched weakness. \textit{J. Pure Appl. Algebra} \textbf{216} (2012), no. 8-9, 1807--1822.

\bibitem[LR16]{LR16}
Stephen Lack and Ji\v{r}\'{\i} Rosick\'{y}. Homotopy locally presentable enriched categories. \textit{Theory Appl. Categ.} \textbf{31} (2016), no. 25, 712--754.

\bibitem[MR14]{cellular}
M. Makkai and J. Rosický. Cellular categories. \textit{J. Pure Appl. Algebra} \textbf{218} (2014), no. 9, 1652--1664.

\bibitem[Men12]{menni}
Matías Menni. Bimonadicity and the explicit basis property. \textit{Theory Appl. Categ.} \textbf{26} (2012), No. 22, 554--581.

\bibitem[SS03]{schwede}
Stefan Schwede and Brooke Shipley. Equivalences of monoidal model categories. \textit{Algebr. Geom. Topol.} \textbf{3} (2003), 287--334.

\bibitem[Shu12]{Shu12}
Michael A.\ Shulman. Not every pseudoalgebra is equivalent to a strict one. \textit{Adv. Math.} \textbf{229} (2012), no. 3, 2024--2041.

\end{thebibliography}
\end{document}